\pdfoutput=1

\newcommand{\showdate}{false}

\documentclass[reqno,a4paper]{amsart}
\usepackage{amssymb}
\usepackage[utf8]{inputenc}
\usepackage[pdftitle={Exotic G2-manifolds},hidelinks]{hyperref}
\usepackage{a4wide}
\usepackage{enumitem}
\usepackage{microtype}
\usepackage{ifthen}
\usepackage{comment}
\usepackage{color}
\usepackage{calc}
\usepackage{setspace}
\RequirePackage{xspace}
\RequirePackage{booktabs}
\RequirePackage{multirow}
\RequirePackage{array}
\usepackage{tikz}

\newcommand{\ignore}[1]{}
\newcolumntype{C}[1]{>{\centering\let\newline\\\arraybackslash\hspace{0pt}}p{#1}}

\DeclareMathOperator{\im}{Im}

\DeclareMathOperator{\rk}{rk}

\newcommand{\ie}{\emph{i.e.} }
\newcommand{\eg}{\emph{e.g.} }
\newcommand{\cf}{\emph{cf.} }
\newcommand{\sign}{\sigma}
\newcommand{\kd}{\Sigma}
\newcommand{\hdg}{h}

\newcommand{\bl}{G}
\newcommand{\bm}{H}

\newcommand{\dcov}{p}

\newcommand{\sm}[1]{\left(\begin{smallmatrix} #1 \end{smallmatrix} \right)}
\newcommand{\cvec}[2]{\sm{#1 \\ #2}}
\newcommand{\rvec}[2]{\sm{\!#1 & #2\!}}
\newcommand{\fbb}{\mathcal{Z}}
\newcommand{\sff}{\mathcal{Y}}
\newcommand{\cala}{\mathcal{A}}
\newcommand{\calo}{\mathcal{O}}
\newcommand{\cali}{\mathcal{I}}
\newcommand{\call}{\mathcal{L}}
\DeclareMathOperator{\Amp}{Amp}

\newcommand{\yy}{\textrm{yes}}
\newcommand{\nn}{\textrm{no}}

\newcommand{\hk}{hyper-Kähler\xspace}
\DeclareMathAlphabet{\matheur}{U}{eur}{m}{n}
\newcommand{\hkr}{\matheur{r}}
\newcommand{\kclass}{\matheur{k}}

\DeclareMathOperator{\Pic}{Pic}
\newcommand{\PP}{\mathbb{P}}

\DeclareMathOperator{\Tor}{Tor}
\DeclareMathOperator{\gd}{gd}
\newcommand{\EK}{\mu}
\newcommand{\disc}{\Delta}
\newcommand{\cdisc}{\overline\Delta}
\newcommand{\cw}{W}
\newcommand{\iw}{\mathring W}

\newcommand{\mmod}{\!\!\mod}

\newcommand{\rbox}[2]{\makebox[\widthof{$#1$}-#2]{$#1$}}

\newcommand{\half}{{\textstyle\frac{1}{2}}}

\newcommand{\bbz}{\mathbb{Z}}
\newcommand{\Z}{\mathbb{Z}}

\newcommand{\C}{\mathbb{C}}
\newcommand{\bbq}{\mathbb{Q}}
\newcommand{\bbr}{\mathbb{R}}
\newcommand{\R}{\mathbb{R}}
\newcommand{\bbc}{\mathbb{C}}

\newcommand{\bbo}{\mathbb{O}}
\newcommand{\bbp}{\mathbb{P}}

\newcommand{\into}{\hookrightarrow}

\newcommand{\Num}[1]{\textup{Num}\left(#1\right)}

\newcommand{\gtstr}{$G_{2}$\nobreakdash-\hspace{0pt}structure}

\newcommand{\gtmfd}{$G_{2}$\nobreakdash-\hspace{0pt}manifold}

\newcommand{\dM}{\wh d}
\newcommand{\wt}[1]{\widetilde #1}

\newcommand{\gen}[1]{\langle#1\rangle}

\numberwithin{equation}{section}
\newtheorem{thm}{Theorem}[section]
\newtheorem{prop}[thm]{Proposition}
\newtheorem{lem}[thm]{Lemma}

\newtheorem{cor}[thm]{Corollary}

\theoremstyle{definition}
\newtheorem{defn}[thm]{Definition}

\newtheorem{constr}[thm]{Construction}
\theoremstyle{remark}
\newtheorem{rmk}[thm]{Remark}
\newtheorem*{rmk*}{Remark}

\setlist{leftmargin=*}
\newcommand{\wh}{\widehat}

\newcommand{\pc}{\wh p}
\newcommand{\spc}{$\textrm{spin}^\textrm{c}$\xspace}

\newcommand{\Spin}{\textup{Spin}}

\title{Exotic $G_2$-manifolds}

\author[D. Crowley]{Diarmuid Crowley}
\address{School of Mathematics and Statistics,
University of Melbourne,
Parkville, VIC, 3010, Australia}
\email{dcrowley@unimelb.edu.au}
\author[J. Nordström]{Johannes Nordström}
\address{Department of Mathematical Sciences,
University of Bath,
Bath BA2 7AY, UK}
\email{j.nordstrom@bath.ac.uk}

\begin{document}

\begin{abstract}
We exhibit the first examples of closed 7-dimensional Riemannian
manifolds with holonomy $G_2$ that are homeomorphic but not diffeomorphic.
These are also the first examples of closed Ricci-flat manifolds
that are homeomorphic but not diffeomorphic.
The examples are generated by applying the twisted connected sum construction
to Fano 3-folds of Picard rank 1 and 2.
The smooth structures are distinguished by the generalised Eells-Kuiper
invariant introduced by the authors in a previous paper.  %
\end{abstract}

\ifthenelse{\boolean{\showdate}}{\date{\today}}{}
\maketitle

\ifthenelse{\boolean{\showdate}}{\vspace{-0.8\baselineskip}}{}

\vspace{-5mm plus 10mm}

\section{Introduction}

Given a type of special geometric structure, it is often interesting to ask:
do there exist manifolds with such structures that are homeomorphic but not
diffeomorphic? In this paper we consider the case of Riemannian
metrics with holonomy $G_2$ on closed manifolds of dimension 7.
The Lie group $G_2$ can be described as the automorphism group of the octonion
algebra $\bbo$, and its natural action on $\im \bbo \cong \R^7$
appears as an exceptional case in Berger's classification of Riemannian
holonomy \cite{berger55}.
Metrics with holonomy $G_2$ are always Ricci-flat \cite{bonan66}.

A general strategy to address the question of the first paragraph
is to apply a 
smooth classification theorem
to a plentiful supply of examples for which the classifying invariants
are computable. In this paper we make use of the ``twisted connected sum''
construction of closed \gtmfd s introduced by Kovalev \cite{kovalev03};
it was shown in~\cite{g2m} that this
construction yields large numbers of closed \gtmfd s that are
\emph{2-connected} (\ie~the homotopy groups $\pi_1$ and $\pi_2$ are trivial)
with torsion-free cohomology, and how to compute the
invariants required to apply homeomorphism classification results of
Wilkens \cite{wilkens72}.

The diffeomorphism classification of 2-connected 7-manifolds was recently
completed in \cite{7class}, which in particular introduced a generalised
Eells-Kuiper invariant that distinguishes
all the different smooth
structures on the same closed 2-connected topological spin 7-manifold.
While this invariant can be difficult to compute for
interesting examples of manifolds, in the present paper we show how
to compute it for twisted connected sums, and use that to identify
examples of closed 2-connected manifolds with holonomy $G_2$ that are
homeomorphic but not diffeomorphic.

Using the diffeomorphism classification, the manifolds
can be described explicitly as follows.
Real vector bundles of rank 4 over $S^4$ are classified by their Euler class
$e$ and first Pontrjagin class $p_1$ in $H^4(S^4) \cong \Z$. Let $N$
and $\Sigma_{\rm Mi}$ 
be the total space of the unit sphere bundle in the vector bundle with
$(e,p_1) =(0, 16)$ and $(1,6)$, respectively.
Then $\Sigma_{\rm Mi}$ is an exotic 7-sphere; indeed
$\Sigma_{\rm Mi}$ and $S^7$ were among 
the first discovered 
examples of homeomorphic but non-diffeomorphic
manifolds (Milnor \cite{milnor56}).
Meanwhile for any $k \geq 1$ the
connected sum of $k$ copies of $N$ is a closed 2-connected 7-manifold
with $b_3(N^{\#k}) = k$ and torsion-free cohomology.
The manifolds $N^{\#k}$ and $N^{\#k}\#\Sigma_{\rm Mi}$ are homeomorphic
but not diffeomorphic (and in fact these are the only two diffeomorphism types
with that underlying homeomorphism type).

\begin{thm}
\label{thm:elementary}
For $k = 89$ and $101$, both $N^{\#k}$ and $N^{\#k} \# \Sigma_{\rm Mi}$ admit a Riemannian
metric with holonomy $G_2$.
\end{thm}

To the best of our knowledge, these are
also the first examples of closed Ricci-flat manifolds (of any dimension)
that are homeomorphic but not diffeomorphic.

Since the paper is primarily devoted to the topological analysis of a
particular class of examples of \gtmfd s, it makes practically no use of
results about $G_2$ geometry in general. Background on the definition of
$G_2$ and Riemannian holonomy can be found \eg in the books by Salamon
\cite{salamon89} or Joyce \cite{joyce00}. The main technical work of the paper
is to examine in detail the properties of some examples of Fano 3-folds and
their anticanonical divisors.

\subsection{Twisted connected sums}

There are two known sources of examples of closed \gtmfd s. The first examples
were constructed by Joyce in 1995 by desingularising quotients of
flat tori~\cite{joyce96-I}. In this paper we will make use of the later
twisted connected sum construction. While this can be used to produce a large
number of examples, it is still not known whether there exist infinitely many
different topological types of closed 7-manifolds that admit holonomy
$G_2$ metrics.

A \emph{Fano manifold} is a smooth projective variety with ample anticanonical
bundle, or in more differential-geometric terms, a closed complex manifold
whose first Chern class is a Kähler class. They have been studied extensively,
and in complex dimension 3 they have been classified by
Iskovskih \cite{iskov77, iskov78, iskov79} and
Mori-Mukai \cite{MM1, MM4}. %

Given a pair of Fano 3-folds $Y_+, Y_-$ with smooth anticanonical K3 divisors
$\Sigma_\pm \subset Y_\pm$ and a \emph{matching} diffeomorphism
$\hkr : \Sigma_+ \to \Sigma_-$ (Definition \ref{def:match}), the twisted
connected sum construction yields a closed simply-connected 7-manifold $M$ with
metrics of holonomy $G_2$. The procedure is summarised
in \S \ref{subsec:tcs_def} and  \S \ref{subsec:block_from_fano}.
Part of the usefulness of the twisted connected sum construction is that
many geometric and topological features of the resulting \gtmfd s can
be understood in terms of the relatively well-known algebraic input data.
On the other hand, the challenge is that a thorough understanding of
the algebraic data is required to find any matchings at all.

We categorise the matching as either \emph{perpendicular} or
\emph{non-perpendicular}, depending on the action of $\hkr$ on the images of
$H^2(Y_\pm)$ in $H^2(\Sigma_\pm)$ (Definition \ref{def:config}).
It is shown in \cite{g2m} that for most pairs $\sff_+, \sff_-$ among the 105
deformation types of Fano 3-folds, general deformation theory results make
it possible to find a perpendicular matching of some %
$Y_\pm \in \sff_\pm$ resulting in a 2-connected twisted connected sum \gtmfd.
As we explain below, such perpendicular matchings can never be homeomorphic
without being diffeomorphic.

Whether there is any non-perpendicular matching of a pair of members of
$\sff_\pm$ is in general a more difficult question, which has not previously been studied systematically. There are necessary
conditions of a lattice-arithmetical nature, but, as we discuss in
\S \ref{sec:match}, showing that matchings exist also requires some detailed
information about the deformation theory of anticanonical divisors
in~$\sff_\pm$, which needs to be worked out separately for each individual
deformation type of Fanos.

\subsection{The classifying invariants} 

Let us recall the relevant 
smooth classification results.
Given a closed 2-connected 7-manifold $M$, two obvious topological invariants
are its cohomology ring $H^*(M)$ and its spin characteristic class
$p_M \in H^4(M)$
(satisfying $p_1(M) = 2p_M$). If $H^4(M)$ is torsion-free, then this
data can be reduced to the third Betti number $b_3(M)$ and the greatest
integer $d(M)$ dividing $p_M$ in $H^4(M)$ (we set $d = 0$ if $p_M = 0$).
In fact, the pair $(b_3(M), d(M))$ classifies such $M$ up to homeomorphism
(by Wilkens \cite{wilkens72}, see also Theorem \ref{thm:2c7m}).

In \cite{7class}, we introduced the \emph{generalised Eells-Kuiper invariant}
of a closed spin 7-manifold $M$. If $H^4(M)$ is torsion-free then this
invariant reduces to a constant
\[ \EK(M) \in \Z/\dM , \]
and distinguishes between $\dM := \gcd\big(28, \Num{\frac{d}{4}}\big)$
different diffeomorphism classes of smooth structures on the topological
manifold underlying $M$ (where $\Num{\frac{a}{b}} := \frac{a}{\gcd(a,b)}$).
If $p_M = 0$ (so $\dM = 28$) then $\EK(M)$ coincides with the invariant
introduced by Eells and Kuiper \cite{eells62}, which in particular
distinguishes between the 28 classes of smooth structures on $S^7$.

In \S \ref{subsec:EK_def} we recall how $\EK(M)$ can be defined in terms of a
\spc coboundary of $M$. The challenge with this definition is that while the 
existence of a suitable coboundary is guaranteed, there is
no algorithm for finding one, especially not one with a simple enough
description that evaluating the formula \eqref{eq:ek_def} is tractable.
However, we are able to construct explicit \spc coboundaries of twisted
connected sums, and in \S \ref{sec:ek_tcs} we use those to express the
generalised Eells-Kuiper invariant of a twisted connected sum in terms of
data for the Fanos used and the matching.
In particular, it turns out that any perpendicularly matched
twisted connected sum has $\EK = 0$ (Corollary \ref{cor:perp0}).

\subsection{The main examples}

To have any chance of obtaining homeomorphic but non-diffeomorphic twisted
connected sums we must therefore search for non-perpendicular %
matchings. In that case, %
both Fanos used must have Picard rank $\geq 2$, \cf Remark
\ref{rmk:arith}. (The Picard group $\Pic Y$ of a Fano 3-fold $Y$
coincides with its integral second cohomology $H^2(Y)$, so the Picard rank
simply means its second Betti number $b_2(Y)$.)
We are therefore led to study systematically the possible
matchings of Fanos with Picard rank 2, and in this setting we can obtain
decisive results.

\pagebreak[3]
\begin{thm}\hfill
\label{thm:main}
\begin{enumerate}
\item Any twisted connected sum $M$ of Fano 3-folds of Picard rank 1 or 2 has
$H^4(M)$ torsion-free.
\item
There are precisely six (unordered) pairs $\sff_+, \sff_-$ of deformation
types of Fano 3-folds of Picard rank 2 with members that can be matched
in the sense of Definition \ref{def:match} in such a way that the resulting
twisted connected sum $M$ has $\EK(M) \not= 0$.
\item
Each of those six pairs gives rise to a single diffeomorphism type of $M$
with $\EK(M) \not= 0$; those $M$ are all 2-connected.
\item
In total, they realise four distinct diffeomorphism types of $M$ with
$\EK(M) \not= 0$.
\item
Precisely two of those are homeomorphic to some twisted connected sum $M'$ of
Fano 3-folds of Picard rank $\leq 2$ such that $\EK(M') = 0$.
\end{enumerate}
In particular, we obtain two pairs $(M,M')$ of manifolds that are homeomorphic
but not diffeomorphic and both admit metrics with holonomy $G_2$.
These are the manifolds identified in Theorem \ref{thm:elementary}.
\end{thm}

For each of the examples highlighted in Theorem \ref{thm:main}, we list 
in Table \ref{table:mainex} the pair of deformation types used, whether the
matching is perpendicular, and the classifying diffeomorphism invariants of the
resulting twisted connected sum.
Recall that a del Pezzo 3-fold is a Fano 3-fold $Y$ whose
anticanonical class $-K_Y \in \Pic Y$ is even.

\vspace{0mm plus 2mm}

\hspace{-5.2cm}
\begin{tabular}{p{1cm}p{9cm}}
\vspace{-3mm}
\[
\refstepcounter{table}
\label{table:mainex}
\begin{array}{c}
\begin{array}{cccccc}
\toprule
\; \sff_+     & \;\sff_-      & \perp & b_3 &  d & \EK \\
\midrule
\ref{it:dp3}  & \ref{it:dp5}  & \yy   & 101 &  8 & 0 \\
\ref{it:mm27} & \ref{it:mm27} & \nn   & 101 &  8 & 1 \\
\ref{it:dp5}  & \ref{it:d16}  & \yy   &  89 &  8 & 0 \\
\ref{it:mm9}  & \ref{it:mm27} & \nn   &  89 &  8 & 1 \\
\ref{it:mm17} & \ref{it:mm27} & \nn   &  89 &  8 & 1 \\
\ref{it:mm9}  & \ref{it:mm9}  & \nn   &  77 &  8 & 1 \\
\ref{it:mm17} & \ref{it:mm9}  & \nn   &  77 &  8 & 1 \\
\ref{it:mm17} & \ref{it:mm17} & \nn   &  77 & 24 & 1 \\
\bottomrule 
\end{array} \\[20.5mm] \textsc{Table \thetable} \end{array} \]
&

\begin{enumerate}[label=\textup{(\alph*)}]
\item \label{it:dp3}
Del Pezzo 3-folds of degree 3, \ie cubic hypersurfaces in $\PP^4$
(Picard rank 1).
\item \label{it:dp5}
Del Pezzo 3-folds of degree 5 (Picard rank 1).
\item \label{it:d16}
Picard rank 1 Fanos of degree 16.
\item \label{it:mm9}
Number 9 in the Mori-Mukai list of Picard rank 2 Fano 3-folds:
$\PP^3$ blown up in a curve of degree 7 and genus~5.
\item \label{it:mm17}
Number 17 in the Mori-Mukai list of Picard rank 2 Fano 3-folds: a smooth
quadric hypersurface in $\PP^4$ blown up in an elliptic curve of degree 5.
\item \label{it:mm27}
Number 27 in the Mori-Mukai list of Picard rank 2 Fano 3-folds:
$\PP^3$ blown up in a twisted cubic curve.
\end{enumerate}
\end{tabular}

\vspace{0mm plus 2mm}
As seen in Table \ref{table:mainex}, the two pairs that are homeomorphic but
not diffeomorphic have ${d = 8}$ and $b_3 = 89$ or $101$,
coinciding with the invariants of the manifolds in
Theorem \ref{thm:elementary} (see \mbox{\cite[Example~5.3]{7class}}).

In \S \ref{sec:blocks} we %
compute detailed topological data
for all 36 types in the Mori-Mukai list of rank 2 Fano 3-folds.
In \S \ref{sec:handcraft} we identify all pairs that satisfy the necessary
arithmetic conditions for existence of a non-perpendicular matching resulting
in a twisted connected sum with $\EK \not= 0$. The only candidate pairs are
among the types \ref{it:mm9}, \ref{it:mm17} and \ref{it:mm27} above.

The key difficulty in finding non-perpendicular matchings is to understand
precisely which K3 surfaces $\kd$ appear as anticanonical divisors in a given
type of Fanos, identifying conditions in terms of the Picard lattice of $\kd$
(\ie $\Pic \kd = H^2(\kd; \Z) \cap H^{1,1}(\kd; \C)$ equipped with the
intersection form).
Having at least reduced our list of candidates, we carry out this intricate
work only for the types \ref{it:mm9}, \ref{it:mm17} and \ref{it:mm27}.
We find in Theorem \ref{thm:nonzero} that non-perpendicular matchings do in
fact exist for each of the six pairs of those types, leading to the examples
with $\EK \not= 0$ above.

We then compare the homeomorphism invariants $(b_3, d)$ of the
realised manifolds with the invariants realised by perpendicular matchings of
the 1378 pairs of rank 1 and 2 Fanos, listed in Table~\ref{table:perp}
of \S\ref{sec:orth}. For two
of the four twisted connected sums with $\EK \not= 0$ we can identify some
perpendicular matching with the same homeomorphism invariants, and two of those
are included in the table above.

In another application of the classification results, we can also exhibit
examples of holonomy $G_2$ metrics on total spaces of smooth $S^1$-bundles,
while the last entry of Table \ref{table:mainex} is an example of a \gtmfd{}
that is a topological $S^1$-bundle but not a smooth one
(Remark \ref{rmk:circle_bundles}).
Further, as a byproduct of our analysis in \S\ref{sec:orth} we identify all
pairs of rank 2 Fanos that can be matched to define twisted connected sums $M$
with $H^2(M) \cong \Z$ (Table \ref{table:orth}).
Such matchings are of interest for the problem of constructing
examples of $G_2$-instantons %
by gluing, \cf Sá Earp and Walpuski \cite{saearp13}, Walpuski \cite{walpuski16}
and Menet, the second author and Sá~Earp~\cite{tcsinst}. 

\subsection{Context}

In dimension 7, the problem of finding special metrics on manifolds that are
homeomorphic but not diffeomorphic has been considered for instance in the
case of Riemannian metrics with positive sectional curvature
(Kreck-Stolz \cite{kreck91}) and
3-Sasakian metrics (Chinburg-Escher-Ziller \cite{chinburg07}). In both cases,
the examples exhibited have finite $H^4$, and the smooth structures can be
distinguished by the ordinary Eells-Kuiper invariant.

All known (irreducible) examples of closed Ricci-flat manifolds have special
holonomy: they are $2n$-manifolds with
holonomy $SU(n)$ ($n \geq 2$), $4n$-manifolds with holonomy
$Sp(n)$ ($n \geq 2$), 7-manifolds with holonomy $G_2$ and
8-manifolds with holonomy $\Spin(7)$.

In real dimension 4, the only smooth manifold with holonomy $SU(2)$ is the K3
surface. In real dimension 6, simply-connected manifolds have unique smooth
structures by Zhubr \cite[Theorem~6.3]{zhubr2000}, 
while there is no general classification result for finite but
non-trivial fundamental group.

Manifolds with holonomy $SU(n)$ or $Sp(n)$ necessarily have $b_2 \geq 1$,
and in real dimension $\geq 8$ the case $b_2 > 1$ is out of reach of current
smooth classification results.
Complete intersections with $c_1 = 0$ provide examples with holonomy $SU(n)$ and $b_2 = 1$, 
many of which have been smoothly classified by Traving \cite{traving85} 
(see also \cite[Theorem~A]{kreck99});
however Traving's results imply that homeomorphic complete intersections with $c_1 = 0$
are diffeomorphic in real dimension $\geq 6$.

Joyce \cite[Theorem 15.4.3, 15.5.2, 15.5.6,
15.6.2 \& 15.7.2]{joyce00} provides five examples of $8$-manifolds of holonomy
$\Spin(7)$ with $b_2 = 0$, and four of these also have $b_3 = 0$.
Computing the torsion in the cohomology of these examples is intricate,
and even if they prove to be 3-connected, then the underlying topological
manifold admits two diffeomorphism classes of smooth structure if and only if 8 
divides the first Pontrjagin class (if 8 does not divide the first Pontrjagin class, then the smooth 
structure is unique up to diffeomorphism) \cite{crowley19}.
However, there is not currently a sufficiently
large supply of such examples to be hopeful of finding exotic pairs of $3$-connected $\Spin(7)$
$8$-manifolds.

Thus 7-manifolds of holonomy $G_2$ are the only kinds of closed Ricci-flat
manifolds where homeomorphic non-diffeomorphic
examples are accessible with the current technology.

In the context of non-Ricci-flat holonomy groups, an early application
of Donaldson invariants was to give examples of closed manifolds with holonomy $U(2)$,
\ie Kähler manifolds of complex dimension~2, that are homeomorphic but
not diffeomorphic \cite{donaldson85}.

\subsection*{Acknowledgements} We thank Alastair Craw, Alessio Corti,
Mark Haskins, Jesus Martinez-Garcia, Matthew Turner and Dominic Wallis for
useful discussions, and the referee for constructive comments.
DC acknowledges the support of the Leibniz Prize of Wolfgang L\"{u}ck, granted
by the Deutsche Forschungsgemeinschaft.
JN thanks the Simons Foundation for its support under the Simons Collaboration
on Special Holonomy in Geometry, Analysis and Physics
(grant \#488631, Johannes Nordström).

\section{Background}

We begin with %
further
explanations of the generalised Eells-Kuiper
invariant and the twisted connected sum construction, in order to 
elucidate
the meaning of the %
main theorem, Theorem~\ref{thm:main}.

\subsection{Spin and spin$^\textrm{c}$ characteristic classes}

Recall that $B\Spin$,
the classifying space for stable spin vector bundles,
is \mbox{$3$-connected} with $\pi_4(B\Spin) \cong \Z$.
It follows that $H^4(B\Spin) \cong \Z$ is infinite cyclic.  A generator is
denoted $\pm \frac{p_1}{2}$ and the notation is justified since for the
canonical map $\pi \colon B\Spin \to BSO$ we have $\pi^*p_1 = 2 \frac{p_1}{2}$,
where $p_1$ is the first Pontrjagin class.
Given a spin manifold $X$ we write
\[ p_X : = \frac{p_1}{2}(TX) \in H^4(X) .\]
A \spc structure on a vector bundle $E$ has an associated complex
line bundle $L$ such that $E \oplus L$ is spin. Given a \spc structure on the
tangent bundle of $X$ we can therefore define characteristic classes
$z := c_1(L) \in H^2(X)$, and
$\pc_X := \frac{p_1}{2}(TX \oplus L) \in H^4(X)$.
Then $2\pc_X = p_1(X) + z^2$.
(If $X$ is spin then the induced \spc structure of course has $L$ trivial,
and $\pc_X = p_X$.)

\enlargethispage{0.7\baselineskip}

\begin{lem}[{\cf \cite[2.39--2.40]{7class}}]\hfill
\label{lem:px}
\begin{enumerate}
\item \label{it:acx} If $TX$ has an almost complex structure then
$\pc_X = -c_2(X) + c_1(X)^2$.
\item $\pc_X = w_4(X) + w_2(X)^2 \mmod 2$.
\item \label{it:wu}
Suppose $X$ is closed.
\begin{itemize}
\item If $\dim X \leq 7$ then $\pc_X$ is even.
\item If $\dim X = 8$ then $\pc_X$ is characteristic for the intersection form
of $X$, \ie $x^2 = x \cup \pc_X \mmod 2$ for any $x \in H^4(X)$.
\end{itemize}
\end{enumerate}
\end{lem}

\subsection{The generalised Eells-Kuiper invariant}
\label{subsec:EK_def}

We now describe the %
invariant $\EK(M)$ of a closed spin 7-manifold $M$, in the special case when
$H^4(M)$ is torsion-free. As in the introduction, let $d$ be the greatest
integer dividing $p_M$ (so $d$ is even by Lemma \ref{lem:px}\ref{it:wu}),
and $\dM := \gcd\big(28, \Num{\frac{d}{4}}\big)$.

Let $\cw$ be a \spc 8-manifold with $\partial \cw = M$.
To use $\cw$ to compute the generalised Eells-Kuiper invariant $\EK(M)$ we need
an element $n \in H^4(\cw)$ such that the image of $d n$ in $H^4(M)$ equals
$p_M$ \mbox{\cite[\S 2.7]{7class}}.
Then
\begin{equation}
\label{eq:ek_def}
\EK(M) :=
\frac{\wh \alpha^2 - \sign(\cw)}{8}
- \frac{5\bar z^2 \pc_\cw}{12} + \frac{z^4}{4} \in \Z/\dM,
\end{equation}
where $\wh \alpha := \pc_\cw - d n \in H^4(W)$ for $n$ as above, $\sign(W)$
denotes the signature of the intersection form of $W$, and
$\bar z \in H^2(\cw,M)$ is any pre-image of $z$ ($\bar z^2 \in H^4(\cw,M)$
is independent of the choice of $\bar z$). The integrals of $\wh \alpha^2$
and $z^4$ make sense since they are squares of classes that vanish on the
boundary. That $\EK(M)$ is independent of the choice of $\cw$ is a consequence
of the index formula for the Dirac operator on a closed \spc 8-manifold.

When $M$ is in addition 2-connected, there are precisely $\dM$ different smooth
structures on the topological manifold underlying $M$, and they are
distinguished by $\EK(M)$.

\begin{thm}[\cf {\cite[Theorems 1.2 \& 1.3]{7class}}]
\label{thm:2c7m}
Let $M_0$ and $M_1$ be closed 2-connected 7-manifolds such that $H^4(M_i)$
are torsion-free. Then
\begin{enumerate}
\item $M_0$ and $M_1$ are homeomorphic if and only if $b_3(M_0) = b_3(M_1)$
and $d(M_0) = d(M_1)$.
\item $M_0$ and $M_1$ are diffeomorphic if and only if in addition
$\EK(M_0) = \EK(M_1)$.
\end{enumerate}
\end{thm}

\subsection{Definition of twisted connected sums}
\label{subsec:tcs_def}

We now explain the construction of the twisted connected sum of a matching
pair of building blocks, using the set-up from \cite[\S 3]{g2m}.
Like in the original application of the construction by
Kovalev \cite{kovalev03}, the present paper uses building blocks 
obtained from Fano 3-folds (in our case of rank 1 or 2) by a procedure
explained in \S \ref{subsec:block_from_fano}. The matching problem is
elaborated on in \S \ref{sec:match}.

\begin{defn}
Let $Z$ be a non-singular algebraic 3-fold and $\Sigma \subset Z$ a
non-singular K3 surface. Let $N$ be the image of $H^2(Z) \to H^2(\kd)$.
We call $(Z,\kd)$ a \emph{building block} if %
\begin{enumerate}[leftmargin=*]
\item the class in $H^2(Z)$ of the anticanonical line bundle $-K_Z$
is indivisible;
\item
$\kd\in |{-}K_Z|$ (\ie $\Sigma$ is an anticanonical divisor),
and there is a  
projective morphism $f\colon Z\to \PP^1$ with $\kd=f^\star (\infty)$;
\item the inclusion $N\hookrightarrow H^2(\kd)$ is primitive, that is,
  $H^2(\kd)/N$ is torsion-free;
\label{N:primitive}
\item the group $H^3(Z)$---and thus also $H^4(Z)$---is torsion-free.
\end{enumerate}
We call $N$, equipped with the restriction of the intersection form on
$H^2(\kd)$, the \emph{polarising lattice} of the block. (Because $H^{2,0}(Z)$
is automatically trivial, $N \subseteq H^{1,1}(\kd)$ \cite[Lemma 3.6]{g2m},
so that $\kd$ is `$N$-polarised'.)
\end{defn}

\begin{rmk*}
The class $[\kd] = -K_Z \in H^2(Z)$ is always in the kernel
of $H^2(Z) \to H^2(\kd)$.
In this paper we will only consider blocks where $[\kd]$ in fact generates
the kernel.
\end{rmk*}

\begin{defn}
\label{def:match}
Let $Z_\pm$ be complex 3-folds, $\kd_\pm \subset Z_\pm$ smooth anticanonical
K3 divisors and $\kclass_\pm \in H^2(Z_\pm; \bbr)$ Kähler classes.
Let $\Pi_\pm \subset H^2(\kd_\pm;\bbr)$ denote the type $(2,0)+(0,2)$ part.
We call a diffeomorphism $\hkr \colon {\kd_+ \to \kd_-}$ a \emph{matching}
of $(Z_+, \kd_+, \kclass_+)$ and $(Z_-, \kd_-, \kclass_-)$ if
$\hkr^* \kclass_- \in \Pi_+$ and $(\hkr^{-1})^* \kclass_+ \in \Pi_-$,
while $\hkr^* \Pi_- \cap \Pi_+$ is non-trivial.

We also say that $\hkr : \kd_+ \to \kd_-$ is a matching of
$Z_+$ and $Z_-$ if there are Kähler classes $\kclass_\pm$ so that the above
holds.
\end{defn}

Given a building block $(Z, \kd)$, let $\disc \subset \PP^1$ be an open disc
that is a trivialising
neighbourhood of~$\infty$ for the fibration $f : Z \to \PP^1$, and
$U := f^{-1}(\disc)$. Then $V := Z \setminus U$ is a manifold with boundary
diffeomorphic to $S^1 \times \kd$.

\enlargethispage{0.2\baselineskip}

\begin{constr}[Twisted connected sum]
\label{constr:tcs}
Given a pair of building blocks $(Z_\pm, \kd_\pm)$ with a matching $\hkr$,
their \emph{twisted connected sum} $M$ is the smooth 7-manifold defined by
gluing $S^1 \times V_+$ and $S^1 \times V_-$ by the diffeomorphism
\begin{equation}
\label{eq:gluemap}
\begin{aligned}
S^1 \times S^1 \times \kd_+ & \to S^1 \times S^1 \times \kd_- , \\
(u, v, x) &\mapsto (v, u, \hkr(x))
\end{aligned}
\end{equation}
of their boundaries.
\end{constr}

\pagebreak[2]

\begin{thm}
\label{thm:g2glue}
The twisted connected sum $M$ admits metrics with holonomy $G_2$.
\end{thm}

\begin{proof}[{Proof sketch \textup{(\cf \cite[Corollary 6.4]{g2m})}}.]
By \cite[Theorem D]{hhn}, the interiors of the manifolds $V_\pm$ admit metrics
with holonomy $SU(3)$ that are asymptotically cylindrical, \ie in a collar
neighbourhood $\cong \bbr \times S^1 \times \kd_\pm$ of the boundary they are
close to a product cylinder metric.
Then $V_\pm$ admit parallel $SU(3)$-structures, defining torsion-free
product \gtstr s on $S^1 \times V_\pm$.
The definition of what it means for $\hkr$ to be a matching ensures that
the $SU(3)$-structures can be chosen so that $\hkr$ is a `\hk rotation'
of the $SU(2)$-structures defining the asymptotic limit.
According to Kovalev \cite[Theorem 5.34]{kovalev03} the map \eqref{eq:gluemap}
can then be used to glue together the product \gtstr s on $S^1 \times V_+$ and
$S^1 \times V_-$ to a torsion-free \gtstr{} on $M$.
\end{proof}

\subsection{Topology of twisted connected sums}
\label{subsec:tcs_top}

Let $(Z_\pm, \kd_\pm)$ be a pair of building blocks with a matching
$\hkr : \kd_+ \to \kd_-$. We think of $\hkr$ as allowing us to identify
both $\kd_+$ and
$\kd_-$ with a standard smooth K3 surface $\kd$. Letting $L := H^2(\kd)$,
we can then think of the polarising lattices $N_\pm$ of the two blocks as
a pair of sublattices of $L$.

Let $T_\pm$ denote the orthogonal complement of $N_\pm$ in $L$.
The following result summarises the cohomology of the twisted connected
sum $M$.

\begin{thm}[{\cite[Theorem 4.8]{g2m}}]
Suppose $M$ is a twisted connected sum of building blocks $(Z_\pm, \kd_\pm)$
such that the kernel of $H^2(Z_\pm) \to H^2(\kd_\pm)$ is generated
by $[\kd_\pm]$. Then
\label{thm:tcs_H}
\hfill
  \begin{enumerate}
  \item \label{it:h1m} $\pi_1(M)=0$ and $H^1(M)=0$;
  \item \label{it:h2m}
    $H^2(M) \cong N_+\cap N_-$;
  \item \label{it:h3m}
    $H^3(M) \cong \Z \oplus (L/_{N_+ + N_-}) \oplus (N_-\cap T_+)
    \oplus (N_+\cap T_-)\oplus H^3(Z_+)\oplus H^3(Z_-)$;
  \item \label{it:h4m}
    $H^4(M) \cong H^4(\kd) \oplus (T_+\cap T_-)\oplus
     (L/_{N_- + T_+})\oplus (L/_{N_+ + T_-}) \oplus H^3(Z_+) \oplus H^3(Z_-)$.
  \end{enumerate}
\end{thm}

\noindent
This implies in particular that if the matching $\hkr$ is perpendicular
(\ie if $N_+$ is perpendicular to~$N_-$ in $L$, \cf Definition \ref{def:config})
then $H^4(M)$ is torsion-free, and
\begin{equation}
b_2(M) = 0, \quad
b_3(M) = b_3(Z_+) + b_3(Z_-) + 23 .
\end{equation}
If in addition the perpendicular direct sum $N_+ \perp N_- \subset L$ is
primitive then $M$ is 2-connected. 

We need to devote some extra attention to describing $H^4(M)$, since that
contains the spin characteristic class $p_M$, whose greatest divisor
$d(M)$ is one of our classifying invariants from Theorem \ref{thm:2c7m}.
Let
\begin{equation}
\label{eq:h0z}
H^4(Z_+) \oplus_0 H^4(Z_-) \subset H^4(Z_+) \oplus H^4(Z_-)
\end{equation}
be the subgroup of pairs with equal image in $H^4(\kd)$.
We define a homomorphism
\begin{equation}
\label{eq:defy}
Y : H^4(Z_+) \oplus_0 H^4(Z_-) \to H^4(M)
\end{equation}
as follows (\cf \cite[Definition 4.13]{g2m}).
For $([\alpha_+], [\alpha_-]) \in H^4(Z_+) \oplus_0 H^4(Z_-)$, let $[\beta]$ be
their common image in $H^4(\kd)$. 
Then we may choose the cocycle
$\alpha_\pm \in C^4(Z_\pm;\bbz)$ so that the restriction of $\alpha_\pm$ to
$\partial V_\pm \cong S^1 \times \kd \subset Z_\pm$ is the pull-back
of $\beta$ by projection to the $\kd$ factor.
Then the pull-backs of $\alpha_\pm$ to $S^1 \times V_\pm$ 
patch to a cocycle on $M$ under the gluing \eqref{eq:gluemap},
and we may set $Y([\alpha_+], [\alpha_-])$ to be the
class represented by that cocycle.

Let
\begin{equation}
\label{eq:flat}
\flat^{\pm} : N_\mp \to N_\pm^*
\end{equation}
denote the homomorphism induced by the intersection form of $L$,
and $N_\mp' \subseteq N_\pm^*$ the image of~$\flat^{\pm}$.
Since $N_\pm^* \subset H^4(Z_\pm)$, we can also regard
$N'_\mp$ as a subset of $H^4(Z_+) \oplus_0 H^4(Z_-)$.

\begin{lem}[{\cite[Lemma 4.14]{g2m}}]
\label{lem:imy}
The image of the map $Y$ defined in \eqref{eq:defy} is the direct summand 
$H^4(\kd) \oplus (L/_{N_- + T_+})\oplus (L/_{N_+ + T_-})$ of $H^4(M)$,
and the kernel is $N'_- \oplus N'_+ \subset H^4(Z_+) \oplus_0 H^4(Z_-)$.
\end{lem}

\begin{rmk}
\label{rmk:h4tor}
The image of $Y$ contains all the torsion of $H^4(M)$, which we can identify
as
\begin{equation}
\label{eq:h4tor}
\Tor H^4(M) \cong \Tor N_+^*/N'_- \oplus \Tor N_-^*/N'_+ .
\end{equation}
\end{rmk}

\begin{prop}[{\cite[Proposition 4.20/Remark 4.21]{g2m}}]
\label{prop:p1y}
\[ p_M = -Y(c_2(Z_+), c_2(Z_-)) . \]
\end{prop}

Combining the previous two results, the greatest divisor $d(M)$ of $p_M$ can be
determined from $c_2(Z_\pm) \in H^4(Z_\pm)$ (which depends purely on the
blocks) and $N'_\mp \subset H^4(Z_\pm)$ (which depends on the matching).

\section{The generalised Eells-Kuiper invariant of twisted connected sums}
\label{sec:ek_tcs}

The main result of this section is Theorem \ref{thm:ek_tcs},
which gives a formula for the generalised Eells-Kuiper invariant $\EK(M)$ of a
twisted connected sum $M$ in terms of data for the pair of building
blocks used and data for the matching.
We compute $\EK(M)$ using an explicit \spc coboundary.

Throughout this section, let $M$ be the twisted connected sum of
a pair of building blocks $(Z_\pm, \kd_\pm)$, using a matching
$\hkr \colon \kd_+ \to \kd_-$. To simplify the description of the cohomology,
we assume that the kernel
of $H^2(Z_\pm) \to H^2(\kd_\pm)$ is generated by $[\kd_\pm]$.

\setlength\normalparindent{\the\parindent}
\noindent
\begin{tabular}{@{}b{9cm}b{5cm}}
\hspace{\normalparindent}The \spc coboundaries we construct to compute the 
generalised Eells-Kuiper invariant for twisted connected
sum manifolds can be viewed as {\em parametrised plumbing} of trivial
disc bundles. By way of context we recall the plumbing of bundles;
see Browder \cite[V \S2]{browder72}.
If $\pi_0 \colon W_0 \to M_0$ is a \mbox{$D^p$-bundle} over 
a $q$-manifold and 
$\pi_1 \colon W_1 \to M_1$
is a \mbox{$D^q$-bundle} over 
a \mbox{$p$-manifold}, then the \emph{plumbing} of $\pi_0$ and $\pi_1$ is the
manifold $W = (W_0 \sqcup W_1)/{\sim}$ obtained from the disjoint union
of $W_0$ and $W_1$ by trivialising $\pi_0$ over $D^q \subset M_0$ and 
$\pi_1$ over $D^p \subset M_1$, identifying the resulting spaces 
$D^p \times D^q$ and $D^q \times D^p$ by exchanging co-ordinates
and finally smoothing corners.
The case of a pair of trivial $D^1$-bundles over
$S^1$ is illustrated on the right.
&
\begin{tikzpicture}[y=0.80pt, x=0.80pt, yscale=-3.7, xscale=3.7, inner sep=0pt, outer sep=0pt]
\begin{scope}[shift={(-3.26721,-244.17523)}]
  \path[draw=gray,miter limit=4.00,line width=1pt]
    (21.3672,277.1752) circle (0.4233cm);
  \path[draw=gray,miter limit=4.00,line width=1pt]
    (32.4377,276.6514)arc(105.187:187.915:15.000)arc(187.915:270.643:15.000)arc(270.643:353.372:15.000)arc(-6.628:76.100:15.000);
  \path[draw=black,miter limit=4.00,line width=1pt]
    (18.5743,259.4525)arc(188.700:272.400:18.000)arc(272.400:356.100:18.000)arc(-3.900:79.800:18.000);
  \path[draw=black,miter limit=4.00,line width=1pt]
    (24.8356,265.6874)arc(-73.200:9.464:12.000)arc(9.464:92.128:12.000)arc(92.128:174.791:12.000)arc(174.791:257.455:12.000);
  \path[draw=black,miter limit=4.00,line width=1pt]
    (24.6517,259.5780)arc(192.500:273.133:12.000)arc(273.133:353.767:12.000)arc(-6.233:74.400:12.000);
  \path[draw=black,miter limit=4.00,line width=1pt]
    (24.5857,259.4653)arc(-79.700:5.563:18.000)arc(5.562:90.825:18.000)arc(90.825:176.087:18.000)arc(176.087:261.350:18.000);
  \path[draw=black,miter limit=4.00,line width=1pt]
    (32.3354,279.7179)arc(102.943:169.708:18.000);
  \path[draw=black,miter limit=4.00,line width=1pt]
    (32.0864,273.3857)arc(110.900:163.204:12.000);
\end{scope}
\end{tikzpicture}
\end{tabular}

\begin{constr}[Parametric plumbing]
\label{constr:coboundary}
As in Construction \ref{constr:tcs}, we use the tubular neighbourhood 
$U_\pm \cong \disc \times \kd_\pm \subset Z_\pm$ of~$\kd_\pm$;
denote the coordinate on the open disc $\disc$ by $z$.
Consider $B_\pm := \disc \times Z_\pm$, and denote the coordinate on its
open disc factor by $w$.
Form an 8-manifold $\iw$ by gluing $B_+$ and $B_-$ along
$\disc \times U_\pm$, using the map
\begin{align*}
\disc \times \disc \times \kd_+ & \to \disc \times \disc \times \kd_-, \\
(w, z, x) &\mapsto (z, w, \hkr(x)) .
\end{align*}
We can form a smooth compact manifold $W$ with boundary by attaching
to $\iw$ the result of gluing the boundaries
$\partial \cdisc \times \partial \cdisc \times \kd_\pm$ of
$\partial \cdisc \times (Z_\pm \setminus U_\pm)$
by $(w, z, x) \mapsto (z, w, \hkr(x))$. Comparing with~\eqref{eq:gluemap},
we see that the boundary is precisely the twisted
connected sum $M$. Hence $\cw$ is a coboundary of~$M$.
\end{constr}

Since $Z_\pm$ are not spin, $B_\pm$ and $W$ are not spin either.
However, $Z_\pm$ and $B_\pm$ have complex structures, inducing
\spc structures. While the complex structures of $B_+$ and $B_-$ do not agree
on the overlap region, their \spc structures do agree, so $W$ is \spc too.

\subsection{Cohomology}

We compute the integral cohomology of the coboundary $W$ from Construction
\ref{constr:coboundary} by Mayer-Vietoris. More precisely,
we compute $H^k(\iw) \cong H^k(W)$ and $H^k_{cpt}(\iw) \cong H^k(W,M)$ for
$k \leq 4$. One can of course recover
the remaining cohomology groups from Poincar\'e duality and universal
coefficients, but what we care about is a description for $k \leq 4$ that
lets us determine the characteristic classes and the intersection form.

There are obvious homotopy equivalences
$B_+ \cap B_- \simeq \kd$ and $B_\pm \simeq Z_\pm$, so 
there is a long exact sequence
\[ H^{k-1}(\kd) \to H^k(W) \to H^k(Z_+) \oplus H^k(Z_-) \to H^k(\kd) . \]
Hence, using that $H^2(Z_\pm) \to H^2(\kd) = L$ has image $N_\pm$ by definition 
and is assumed to have kernel $\bbz$, we find
\begin{align*}
H^1(W) &= 0 &
H^2(W) & \cong (N_+ \cap N_-) \oplus \bbz^2 \\
H^3(W) & \cong L / (N_+ + N_-) \oplus H^3(Z_+) \oplus H^3(Z_-) &
H^4(W) & \cong H^4(Z_+) \oplus_0 H^4(Z_-), %
\end{align*}
where $H^4(Z_+) \oplus_0 H^4(Z_-)$ is defined as in \eqref{eq:h0z}.
There is a natural short exact sequence
\begin{equation}
\label{eq:ses}
0 \to N_+^* \oplus N_-^* \to H^4(Z_+) \oplus_0 H^4(Z_-) \to \bbz \to 0 .
\end{equation}
We can describe the isomorphism $H^4(Z_+) \oplus_0 H^4(Z_-) \to H^4(W)$
as follows: given $(c_+, c_-) \in H^4(Z_+) \oplus_0 H^4(Z_-)$, choose a cocycle
$\alpha_0 \in C^4(\kd)$ representing the common image of $c_\pm$ in $H^4(\kd)$,
let $\alpha_\pm \in C^4(Z_\pm)$ be cocycles representing $c_\pm$ that equal
the pull-back of $\alpha_0$ on $\disc \times \kd_\pm$, and form a cocycle on
$W$ by patching the pull-backs of $\alpha_\pm$ to $B_\pm$. It is clear
that the composition of this isomorphism with the restriction map
$H^4(\cw) \to H^4(M)$ equals the map $Y$ defined in \eqref{eq:defy}.

\begin{rmk}
\label{rmk:natural}
Only the isomorphism for $H^3(W)$ involves making an arbitrary choice
for a splitting of a short exact sequence---below it will turn out to be
important that the isomorphisms for $H^2(W)$ and $H^4(W)$ are natural. 
\end{rmk}

Similarly, we can compute $H^*(\cw,M)$ from a long exact sequence
\[ H^{k-4}(\kd) \to H^{k-2}(Z_+) \oplus H^{k-2}(Z_-) \to
H^k(\cw,M) \to H^{k-3}(\kd), \]
finding
\begin{align*}
H^1(\cw,M) &= 0 &
H^2(\cw,M) & \cong \bbz^2 \\
H^3(\cw,M) &= 0  &
H^4(\cw,M) & \cong (H^2(Z_+) \oplus H^2(Z_-))/\bbz . %
\end{align*}
Note there is a short exact sequence
\[ 0 \to \bbz \to (H^2(Z_+) \oplus H^2(Z_-))/\bbz  \to N_+ \oplus N_- \to 0. \]

\subsection{Characteristic classes}

As emphasised in Remark \ref{rmk:natural}, the isomorphisms for $H^2(W)$
and $H^4(W)$ above are natural. One consequence of this is that we can pin down
the characteristic classes of $W$ by considering their restrictions to the open
subsets $B_+ = \disc \times Z_+$ and $B_- = \disc \times Z_-$.

The restriction of $w_2(W) \in H^2(W; \Z/2)$
to $B_\pm$ is $w_2(B_\pm)$, which equals the pull-back of $w_2(Z_\pm)$.
That is the mod~$2$ reduction of $c_1(Z_\pm) = [\kd_\pm] \in H^2(Z_\pm)$ and so
$w_2(Z_\pm)$ is the mod~$2$ Poincar\'e dual of $\kd_\pm$ in~$Z_\pm$.
The \spc structures of $B_\pm$ patch up to a unique \spc
structure on $W$: this is essentially just saying that we
can specify an integral pre-image $z \in H^2(W)$ of $w_2(W)$ uniquely by
setting the restriction to $B_\pm$ to be $c_1(Z_\pm)$. 
Note that the restriction of $z^2$ to each of $B_\pm$ is 0; hence
\[ z^2 = 0 \in H^4(W). \]
Similarly we can pin down the \spc characteristic class $\pc_W$.
The restriction to $B_\pm$ is $\pc_{B_\pm}$, which by
Lemma \ref{lem:px}\ref{it:acx} equals $-c_2 + c_1^2$ of the relevant
$U(4)$-structure. Using again that $c_1(Z_\pm)^2 = 0$, we thus have

\begin{lem}
The image of $\pc_W$ in
$H^4(Z_+) \oplus_0 H^4(Z_-)$ is $(-c_2(Z_+), -c_2(Z_-))$.
\end{lem}

The restriction of $\pc_\cw$ to $M$ is $p_M$, which recovers
Proposition \ref{prop:p1y},
and arguably makes for a nicer proof than that in \cite[Proposition 4.20]{g2m}.

Since the normal bundle of $\kd \subset W$ is trivial,
the image of $\pc_W$ in $H^4(\kd) \cong \Z$ is given by $p_\kd = -c_2(\kd) = -24$, so in
view of the short exact sequence \eqref{eq:ses} most of the interesting
information about $\pc_W$ is captured by
the pre-image of $\pc_W \mmod 24$ in $(N_+^* \oplus N_-^*) \otimes \Z/24$.
Similarly, for any building block $Z$, the image of $c_2(Z)$ under
$H^4(Z) \to H^4(\kd)$ is 24, so $c_2(Z) \mmod 24$ has a preimage in
$H^4(Z, \kd; \Z/24) \cong N^* \otimes \Z/24$.
Denote that by
\begin{equation}
\label{eq:c2bar}
\bar c_2(Z) \in N^* \otimes \Z/24 .
\end{equation}
The calculation above implies that $\pc_W \mmod 24$ is determined
by $\bar c_2(Z_\pm)$. %

\begin{cor}
\label{cor:pmod24}
$\pc_W \mmod 24 \in H^4(W; \Z/24) \cong
H^4(Z_+;\Z/24) \oplus_0 H^4(Z_-; \Z/24)$ is the image of
$(-\bar c_2(Z_+), -\bar c_2(Z_-))$
under the map in \eqref{eq:ses}.
\end{cor}

\begin{rmk}
\label{rmk:even}
Since $c_2(Z_\pm)$ are even, so is $\pc_W$.
Therefore Lemma \ref{lem:px}\ref{it:wu} implies that
the intersection form of $W$ must be even.
\end{rmk}

\subsection{Intersection form}
\label{subsec:intersection}

The intersection pairing between $H^4(\cw)$ and $H^4(\cw,M)$ is
simply the natural duality between $H^4(Z_+) \oplus_0 H^4(Z_-)$ and
$(H^2(Z_+) \oplus H^2(Z_-))/\bbz$. To understand the intersection form
we also need to describe the map $H^4(\cw,M) \to H^4(\cw)$.

First note that the $\bbz$ term in
$H^4(\cw,M) \cong (H^2(Z_+) \oplus H^2(Z_-))/\bbz$ is the Poincar\'e dual to
the K3 divisors, and obviously has trivial image in $H^4(\cw)$.
Hence $H^4(\cw,M) \to H^4(\cw)$ factors through the natural map 
$(H^2(Z_+) \oplus H^2(Z_-))/\bbz \to N_+ \oplus N_-$.
Dually, the composition $H^4(\cw) \cong H^4(Z_+) \oplus_0 H^4(Z_-) \to \bbz$
corresponds to restriction to the K3 divisor, which is trivial for any class
with compact support in $\iw$. So the map also factors through the inclusion
$N_+^* \oplus N_-^* \into H^4(Z_+) \oplus_0 H^4(Z_-)$, %
and is
characterised by a homomorphism $N_+ \oplus N_- \to N_+^* \oplus N_-^*$.

The intersection form on each $B_\pm$ is obviously trivial, so there are no
diagonal terms.
In summary, the map $H^4(W,M) \to H^4(M)$ is therefore the composition of
$H^4(W,M) \to N_+ \oplus N_-$,
\begin{equation}
\label{eq:h4push}
\begin{aligned}
N_+ \oplus N_- & \to N_+^* \oplus N_-^* \\
(x_+, \; x_-) &\mapsto
(\flat^+(x_-), \; \flat^-(x_+))
\end{aligned}
\end{equation}
and $N_+^* \oplus N_-^* \to H^4(W)$,
where $\flat^{\pm}$ is the inner product homomorphism as in \eqref{eq:flat}.
The image equals $N_-' \oplus N_+' \subset N_+^* \oplus N_-^*$,
where $N_\mp'$ is the image of $\flat^\pm : N_\mp \to N_\pm^*$ as before.
(This is consistent with the claim from Lemma \ref{lem:imy} that the kernel
of $Y$ is precisely $N_-' \oplus N_+'$, and also with Remark~\ref{rmk:even}.)
Both summands of $N'_- \oplus N'_+$ are isotropic and so $\sigma(W)$, the
signature of $W$ vanishes:
\begin{equation}
\label{eq:signature_W}
\sigma(W) = 0.
\end{equation}

\subsection{Computing the generalised Eells-Kuiper invariant}

Suppose that the twisted connected sum $M$ has $H^4(M)$ torsion-free, so that
the description of the generalised Eells-Kuiper invariant from
\S \ref{subsec:EK_def} is valid.
Let $d$ denote the greatest integer dividing the spin characteristic class
$p_M \in H^4(M)$. Since $M$ contains a K3 with trivial normal bundle,
$d \mid 24$ a priori. Hence $\dM := \gcd(28, \Num{\frac{d}{4}})$ is either
1 or 2, depending on whether $d$ is divisible by 8 or not.
In particular $\EK(M) \in \Z/\dM$ can only possibly be non-trivial if
$d$ is 8 or 24.

We now compute $\EK(M)$ by evaluating \eqref{eq:ek_def} for the parametric
plumbing coboundary of Construction \ref{constr:coboundary}.
The $z^4$ term in \eqref{eq:ek_def} obviously makes no contribution since
$z^2 = 0$, and the signature term vanishes by \eqref{eq:signature_W}.
We can take $\bar z$ to be the sum of the generators of
$H^2_{cpt}(\disc) \subset H^2_{cpt}(B_\pm)$, or equivalently the Poincar\'e
duals of $\{0\} \times Z_\pm$. 
This is because the intersection of $\{0\} \times Z_+$  with $\{0\} \times Z_-$
is $\{0\} \times \kd_-$, which is the Poincar\'e dual of $z_{|B_-}$,
and similarly for $z_{|B_+}$.
Thus $\bar z^2$ is twice the Poincar\'e dual to the K3 divisor (the
intersection of $\{0\} \times Z_+$ and $\{0\} \times Z_-$, generating
the copy of $\bbz$ in $(H^2(Z_+) \oplus H^2(Z_-))/\bbz$).
It follows that 
$\bar z^2 \pc_W = 2 \pc_{W|\kd}$. As above, the triviality of the normal
bundle of $\kd$ implies that $\pc_{W|\kd} = p_{\kd} = -24$, and hence
$\frac{5\bar z^2 \pc_W}{12}$ is divisible by 4.
Since the modulus $\dM$ is 1 or 2, the only possible non-trivial
contribution to the RHS of \eqref{eq:ek_def} is $\frac{\wh \alpha^2}{8}$.

The definition of $d$ means that $\pc_W \in N_-' \oplus N_+' \mmod d$,
where $N_-' \oplus N_+'$ is the image of $H^4(W,M) \to H^4(W)$ identified
in \S \ref{subsec:intersection}.
By Corollary \ref{cor:pmod24}, we can find elements
$[x_\pm] \in {N_\pm \otimes \Z/d}$
such that $\pc_W = (\flat^+([x_-]), \flat^-([x_+])) \mmod d$ from
$\bar c_2(Z_\pm) \in N_\pm^* \otimes \Z/24$ and the configuration of embeddings
$N_\pm \subset L$ of the matching.
We may then take
\[ \wh \alpha = (\flat^+(x_-), \flat^-(x_+)) \in N_-' \oplus N_+' . \]
One pre-image under \eqref{eq:h4push} of this $\wh \alpha$ in $N_+ \oplus N_-$
is $(x_+, \; x_-)$, so $\wh \alpha^2 = 2x_+.x_-$.
Hence we have

\begin{thm}
\label{thm:ek_tcs}
Let $x_\pm \in N_\pm $ such that $\flat^\mp(x_\pm) = \bar c_2(Z_\mp) \mmod d$.
Then
\begin{equation}
\label{eq:eek_claim}
\EK(M) \; = \; \frac{x_+.x_-}{4} \in \Z/\dM .
\end{equation}
\end{thm}

The elements $x_\pm$ must be even because $c_2(Z_\pm)$ is, so the RHS is indeed integral.

If $N_+ \perp N_-$, \ie if the matching used is perpendicular, then $d$ is the
greatest common divisor of $\bar c_2(Z_+)$ and $\bar c_2(Z_-)$, so we can
trivially take $x_\pm = 0$.
Indeed, $H^4(Z_+) \oplus_0 H^4(Z_-) \into H^4(M)$, so since there is no
torsion in $H^3(Z_\pm)$, there is a well-defined
$n := \frac{1}{d} \pc_W \in H^4(Z_+) \oplus_0 H^4(Z_-)$, and that choice
gives $\wh \alpha = 0$.

\begin{cor}
\label{cor:perp0}
$\EK(M) = 0$ for all twisted connected sums obtained by perpendicular matching.
\end{cor}

But if neither of $c_2(Z_\pm)$ are divisible by 4, then
$x_+.x_-$ can be $4 \mmod 8$. For non-perpendicular matchings that can happen
even when $d$ is divisible by 8, in which case $\EK(M)$ is non-zero.

\subsection{Quadratic refinement of the torsion linking form}

Although in this paper we only consider examples of twisted connected
sums with no torsion in $H^4(M)$, let us briefly
use the coboundary $W$ to
analyse in general the torsion linking form of $M$
and its quadratic refinement (\cf \cite[Definition~2.23]{7class}).

Since the %
blocks $Z_\pm$ are assumed to have no torsion in
$H^3(Z_\pm)$, $H^4(W)$ resolves all the torsion in~$H^4(M)$.
\cite[Remark 4.12]{g2m} claims that the two summands $\Tor N_\pm^*/N'_\mp$
of $\Tor H^4(M)$ in \eqref{eq:h4tor} are isotropic for the
torsion linking form, and naturally dual. This can be checked by
considering cocycles representing the classes, but it is nicer to do it by
computing cup products on~$W$. Indeed, this way we can also compute the
refinement of the linking form.
Because $\pc_W$ is even by Remark~\ref{rmk:even}, we can characterise the
quadratic linking family on $\Tor H^4(M)$ as follows: it assigns to the image
in $H^4(M)$ of (the unique) $\half\pc_W$ the discriminant form $q$ of the even
lattice $H^4(\cw,M)/\mathrm{rad}$.

\begin{prop}
$\Tor N_\pm^*/N'_\mp \subseteq \Tor H^4(M)$ is isotropic for the torsion
linking form, and Lagrangian for $q$.
\end{prop}

\begin{proof}
We can write any $c \in \Tor N_+^*/N'_- \subset H^4(M)$
as the restriction of an element of the form 
$(\alpha,0) \in N_+^* \oplus N_-^* \subset H^4(W)$, where $\alpha$ is in the
rational span of $N_-'$. Then $(\alpha, 0)$ is the image under
\eqref{eq:h4push} of some $(0, y) \in (N_+ \oplus N_-) \otimes \bbq$.
This pairs trivially with $(\alpha, 0)$, so $q(c) = 0$.
\end{proof}

\newcommand{\firstblocks}{
\begin{table}[tb]
\vspace{-4mm}
\[
\begin{array}[t]{l c  c  c  c  c} \toprule
Y & r & -K_Y^3 & b_3(Y) & b_3(Z) & \pi_! c_2(Z) \\ \midrule
\PP^3 &  4 & 4^3         &   0 &  66 &  22 \\ 
    Q &  3 & 3^3 \cdot 2 &   0 &  56 &  26 \\ 
  V_1 &  2 & 2^3         &  42 &  52 &  16 \\ 
  V_2 &  2 & 2^3 \cdot 2 &  20 &  38 &  20 \\ 
  V_3 &  2 & 2^3 \cdot 3 &  10 &  36 &  24 \\ 
  V_4 &  2 & 2^3 \cdot 4 &   4 &  38 &  28 \\ 
  V_5 &  2 & 2^3 \cdot 5 &   0 &  42 &  32 \\ 
      &  1 &           2 & 104 & 108 &  26 \\ 
      &  1 &           4 &  60 &  66 &  28 \\ 
      &  1 &           6 &  40 &  48 &  30 \\ 
      &  1 &           8 &  28 &  38 &  32 \\ 
      &  1 &          10 &  20 &  32 &  34 \\ 
      &  1 &          12 &  14 &  28 &  36 \\ 
      &  1 &          14 &  10 &  26 &  38 \\ 
      &  1 &          16 &   6 &  24 &  40 \\ 
      &  1 &          18 &   4 &  24 &  42 \\ 
      &  1 &          22 &   0 &  24 &  46 \\ 
\bottomrule
\end{array} 
\]
\smallskip
\caption{Rank 1 Fano blocks}
\label{table:species1}
\vspace{-4mm plus 20mm}
\end{table}}

\section{Fano-type blocks}
\label{sec:blocks}

We now describe how to construct building blocks for twisted connected sums
from closed Fano 3-folds, and list detailed data for blocks obtained from
Fanos of Picard rank 1 or 2.

\subsection{Construction of building blocks from Fanos}
\label{subsec:block_from_fano}

Recall that a Fano 3-fold is a smooth closed complex 3-fold $Y$ whose
anticanonical bundle $-K_Y$ is ample, or equivalently, $c_1(Y)$ can be
represented by a Kähler form.
Any Fano 3-fold is simply-connected with $H^{2,0}(Y) = 0$,
so $\Pic Y \cong H^2(Y)$.
Together with a natural form, this is a key deformation invariant of $Y$.

\begin{defn}
\label{def:acform}
For a closed complex 3-fold $Y$, the \emph{anticanonical form} on $H^2(Y)$
is the symmetric bilinear form $(D_1, D_2) \mapsto D_1 \cdot D_2 \cdot (-K_Y)$
(where $\cdot$ is the cup product on $H^2(Y)$).

If $\Pic Y \cong H^2(Y)$ we call $\Pic Y$ equipped with the anticanonical form
the \emph{Picard lattice} of $Y$.
\end{defn}

There are 105 deformation types of Fano 3-folds.
Except for two of those, any Fano 3-fold $Y$ has a pencil
in $|{-}K_Y|$ with smooth base locus $C$.

\begin{constr}
\label{constr:block}
Given a Fano 3-fold~$Y$ and $C \subset Y$ a smooth curve that is the base locus
of an anticanonical pencil, 
let $Z$ be the blow-up of $Y$ in $C$.
If $\kd \subset Z$ is the proper transform of a smooth element of said pencil,
then we call $(Z, \kd)$ a \emph{Fano-type building block}.
\end{constr}

\begin{prop}
[{\cite[Propositions 4.24 and 5.7]{cym},
\cf \cite[Proposition 3.17]{g2m}}]
\label{prop:block}
$(Z, \kd)$ of Construction \ref{constr:block} is indeed a
building block. Further
\begin{enumerate}
\item \label{it:polar}
the image of $H^2(Z) \to H^2(\kd)$ is isomorphic to
$H^2(Y)$, and the kernel is generated by $[\kd]$;
\item \label{it:cones}
the image in $H^{1,1}(\kd)$ of the Kähler cone of $Z$ contains the image
of the Kähler cone of $Y$.
\end{enumerate}
\end{prop}

Some of the important data
of the block can be read off immediately from well-known data about $Y$.
For instance \ref{it:polar} implies that the polarising lattice of the
block is isometric to the Picard lattice of $Y$,
and (\cf \cite[Lemma 5.6]{cym})
\begin{equation}
\label{eq:b3}
b_3(Z) = -K_Y^3 + b_3(Y) + 2 .
\end{equation}
Meanwhile, the second Chern class can be described as follows.
The pull-back $\pi^* : H^2(Y) \to H^2(Z)$ of the blow-up map $\pi : Z \to Y$
is injective, and $H^2(Z)$ is the direct sum of $\pi^* H^2(Y)$ and
$\Z c_1(Z)$.
Thus an element of $H^4(Z)$ can be characterised in terms of its product with
$c_1(Z)$ and its image under the Poincar\'e adjoint
$\pi_! : H^4(Z) \to H^4(Y)$ of $\pi^*$, \ie the map characterised by
equality of the intersection products $(\pi_! x)y$ and $x(\pi^*y) \in \Z$
for any $x \in H^4(Z)$ and $y \in H^2(Y)$.
We have $c_2(Z)c_1(Z) = 24$ and
$c_1(Z)^2 = 0$ for any building block, so
for a Fano-type block $c_2(Z)$ is completely determined by
the following result.

\begin{lem}[{\cite[(5-13)]{cym}}]
\label{lem:c2blowup}
For any complex 3-fold $Y$, $C \subset Y$ a smooth curve contained in a
smooth anticanonical divisor and $\pi : Z \to Y$ the blow-up of $Y$ in $C$,
\begin{equation}
\pi_! (c_2(Z) + c_1(Z)^2) = c_2(Y) + c_1(Y)^2.
\end{equation}
\end{lem}

In view of Corollary \ref{cor:pmod24}, the main interesting 
information about $c_2(Z)$ is the mod 24 reduction
$\bar c_2(Z) \in N^* \otimes \Z/24$, which we can thus think of simply as
\begin{equation}
\label{eq:c2blowup}
\bar c_2(Z) = \pi_!c_2(Z) = c_2(Y) + c_1(Y)^2 \mmod 24 .
\end{equation}
When we tabulate data for the %
blocks, we will record $\pi_! c_2(Z)$
rather than $\bar c_2(Z)$ for completeness.

Note that Definition \ref{def:match} makes sense for Fano 3-folds as well
as building blocks.  Proposition \ref{prop:block}\ref{it:cones} %
implies that a matching
of a pair of Fano 3-folds gives rise to a matching
of the resulting Fano-type blocks. We use the phrase
\emph{twisted connected sum of Fanos} (\eg in the introduction) to mean the
twisted connected sum that arises from such a matching of Fano-type blocks.

\firstblocks

\subsection{Rank 1 blocks}

Table \ref{table:species1} summarises the key data of Fano 3-folds of rank 1
and the resulting building blocks (\cf \cite[Table 1]{cym}).
The data included is the index $r$ (\ie the
largest integer such that $-K_Y = rH$ for some $H \in \Pic Y$),
the anticanonical degree $-K_Y^3$, $b_3(Y)$, $b_3(Z)$, and
the pairing of $\pi_!c_2(Z)$ and the positive generator $H \in \Pic Y$
(equivalently, the product of $c_2(Z)$ and $\pi^* H$).

$b_3(Z)$ is simply obtained from the preceding data by \eqref{eq:b3}.
In the rank 1 case, $\pi_! c_2(Z)$ is also easily determined as follows:
For any Fano one has $c_2(Y)(-K_Y) = 24$, so if $-K_Y = rH$
then \eqref{eq:c2blowup} implies that
\begin{equation}
\label{eq:c2onH}
\pi_! c_2(Z)H = \frac{24-K_Y^3}{r} \mod 24 .
\end{equation}
$N$ is generated by the image of $H$. Its self-intersection (with respect to
the quadratic form on $N$) is not included in the table, but it is simply
$\frac{-K_Y^3}{r^2}$.

\subsection{The table of rank 2 Fano blocks}

\newcommand{\mmtable}{
\newlength{\rowskip}
\setlength{\rowskip}{0.3em}

\begin{table}[hp]
\[
\begin{array}[b]{c@{\hspace{8pt}}cccccccc@{\hspace{5pt}}crr} \toprule
\text{MM\#} & -K_Y^3& N & \Delta & -K_Y & \pi_! c_2(Z) & b_3(Y) & b_3(Z)
 & \pi_! c_2(Z)A & \; A^2 & B^2\; & h \;\;\; \\ 
\midrule
1 & 4 & \sm{0 & 1 \\ 1 & 2} & 1 & \cvec{1}{1} & * & 44 & * \\[\rowskip]
2 & 6 & \sm{0 & 2 \\ 2 & 2} & 4 & \cvec{1}{1} & \rvec{12}{18} & 40 & 48 & 30 & 6
& -6 & -0.58\\[\rowskip]
3 & 8 & \sm{0 & 2 \\ 2 & 4} & 4 & \cvec{1}{1} & \rvec{12}{20} & 22 & 32 \\[\rowskip]
4 & 10 & \sm{0 & 3 \\ 3 & 4} & 9 & \cvec{1}{1} & \rvec{12}{22} & 20 & 32 & 34 & 10 & -90 & -0.15 \\[\rowskip]
5 & 12 & \sm{0 & 3 \\ 3 & 6} & 9 & \cvec{1}{1} & \rvec{12}{24} & 12 & 26 & 36 & 12 & -12 & -0.42 \\[\rowskip]
6 & 12 & \sm{2 & 4 \\ 4 & 2} & 12 & \cvec{1}{1} & \rvec{18}{18} & 18 & 32 & 36 & 12 & -4 & 0 \\[\rowskip]
7 & 14 & \sm{0 & 4 \\ 4 & 6} & 16 & \cvec{1}{1} & \rvec{12}{26} & 10 & 26 & 38 & 14 & -56 & 0.19 \\[\rowskip]
8 & 14 & \sm{2 & 4 \\ 4 & 4} & 8 & \cvec{1}{1} & \rvec{18}{20} & 18 & 34 \\[\rowskip]
9 & 16 & \sm{2 & 5 \\ 5 & 4} & 17 & \cvec{1}{1} & \rvec{18}{22} & 10 & 28 & 40 & 16 & -272 & 0.09 \\[\rowskip]
10 & 16 & \sm{0 & 4 \\ 4 & 8} & 16 & \cvec{1}{1} & \rvec{12}{28} & 6 & 24 & 40 & 16 & -16 & 0 \\
&&&&&&&& 52 & 24 & -24 & -0.58 \\[\rowskip]
11 & 18 & \sm{2 & 5 \\ 5 & 6} & 13 & \cvec{1}{1} & \rvec{18}{24} & 10 & 30 & 42 & 18 & -234 & -0.47 \\[\rowskip]
12 & 20 & \sm{4 & 6 \\ 6 & 4} & 20 & \cvec{1}{1} & \rvec{22}{22} & 6 & 28 & 44 & 20 & -4 & 0 \\[\rowskip]
13 & 20 & \sm{2 & 6 \\ 6 & 6} & 24 & \cvec{1}{1} & \rvec{18}{26} & 4 & 26 & 44 & 20 & -30 & 0.26 \\[\rowskip]
14 & 20 & \sm{0 & 5 \\ 5 & 10} & 25 & \cvec{1}{1} & \rvec{12}{32} & 2 & 24 & 44 & 20 & -20 & 0.32 \\
&&&&&&&& 56 & 30 & -30 & -0.26 \\
&&&&&&&& 68 & 40 & -40 & -0.68 \\[\rowskip]
15 & 22 & \sm{6 & 6 \\ 6 & 4} & 12 & \cvec{1}{1} & \rvec{24}{22} & 8 & 32 \\[\rowskip]
16 & 22 & \sm{2 & 6 \\ 6 & 8} & 20 & \cvec{1}{1} & \rvec{18}{28} & 4 & 28 & 46 & 22 & -110 & -0.14 \\[\rowskip]
17 & 24 & \sm{4 & 7 \\ 7 & 6} & 25 & \cvec{1}{1} & \rvec{22}{26} & 2 & 28 & 48 & 24 & -600 & 0.06 \\[\rowskip]
18 & 24 & \sm{0 & 4 \\ 4 & 2} & 16 & \cvec{1}{2} & \rvec{12}{18} & 4 & 30 & 30 & 10 & -40 & 0.68 \\ 
&&&&&&&& 42 & 18 & -72 & -0.17 \\
&&&&&&&& 54 & 24 & -6 & -0.58 \\[\rowskip]
19 & 26 & \sm{4 & 7 \\ 7 & 8} & 17 & \cvec{1}{1} & \rvec{22}{28} & 4 & 32 & 50 & 26 & -442 & -0.61 \\[\rowskip]
20 & 26 & \sm{2 & 7 \\ 7 & 10} & 29 & \cvec{1}{1} & \rvec{18}{32} & 0 & 28 & 50 & 26 & -754 & 0.16 \\[\rowskip]
21 & 28 & \sm{6 & 8 \\ 8 & 6} & 28 & \cvec{1}{1} & \rvec{26}{26} & 0 & 30 & 52 & 28 & -4 & 0 \\[\rowskip]
22 & 30 & \sm{4 & 8 \\ 8 & 10} & 24 & \cvec{1}{1} & \rvec{22}{32} & 0 & 32 & 54 & 30 & -20 & -0.32 \\[\rowskip]
23 & 30 & \sm{8 & 8 \\ 8 & 6} & 16 & \cvec{1}{1} & \rvec{28}{26} & 2 & 34 \\[\rowskip]
24 & 30 & \sm{2 & 5 \\ 5 & 2} & 21 & \cvec{2}{1} & \rvec{18}{18} & 0 & 32 & 36 & 14 & -6 & 0.58 \\
&&&&&&&& 54 & 30 & -70 & -0.51 \\[\rowskip]
25 & 32 & \sm{0 & 4 \\ 4 & 4} & 16 & \cvec{1}{2} & \rvec{12}{22} & 2 & 36 & 34 & 12 & -12 & 0.42 \\
&&&&&&&& 46 & 20 & -20 & -0.32 \\[\rowskip]
26 & 34 & \sm{6 & 9 \\ 9 & 10} & 21 & \cvec{1}{1} & \rvec{26}{32} & 0 & 36 \\[\rowskip]
27 & 38 & \sm{2 & 5 \\ 5 & 4} & 17 & \cvec{1}{2} & \rvec{18}{22} & 0 & 40 & 40 & 16 & -272 & 0.09 \\[\rowskip] 
28 & 40 & \sm{18 & 9 \\ 9 & 4} & 9 & \cvec{1}{1} & \rvec{42}{22} & 2 & 44 \\[\rowskip]
29 & 40 & \sm{0 & 4 \\ 4 & 6} & 16 & \cvec{1}{2} & \rvec{12}{26} & 0 & 42 & 38 & 14 & -56 & 0.19 \\[\rowskip]
30 & 46 & \sm{6 & 6 \\ 6 & 4} & 12 & \cvec{1}{2} & \rvec{26}{22} & 0 & 48 \\[\rowskip]
31 & 46 & \sm{2 & 5 \\ 5 & 6} & 13 & \cvec{1}{2} & \rvec{18}{26} & 0 & 48 & 44 & 18 & -234 & -0.47 \\[\rowskip]
32 & 48 & \sm{2 & 4 \\ 4 & 2} & 12 & \cvec{2}{2} & \rvec{18}{18} & 0 & 50 & 36 & 12 & -4 & 0\\[\rowskip]
33 & 54 & \sm{0 & 3 \\ 3 & 4} & 9 & \cvec{1}{3} & \rvec{12}{22} & 0 & 56 & 34 & 10 & -90 & -0.15 \\[\rowskip]
34 & 54 & \sm{0 & 3 \\ 3 & 2} & 9 & \cvec{2}{3} & \rvec{12}{18} & 0 & 56 & 30 & 8 & -72 & 0.17\\
&&&&&&&& 42 & 14 & -126 & -0.64 \\[\rowskip]
35 & 56 & \sm{2 & 4 \\ 4 & 4} & 8 & \cvec{2}{2} & \rvec{18}{22} & 0 & 58 \\[\rowskip]
36 & 62 & \sm{2 & 5 \\ 5 & 10} & 5 & \cvec{1}{2} & \rvec{18}{34} & 0 & 64 \\[\rowskip]
\bottomrule
\end{array}
\]
\vspace{-4mm}
\caption{Rank 2 Fano blocks}
\label{table:blocks}
\end{table}}

In Table \ref{table:blocks} we collect the following data for building blocks
obtained from rank 2 Fano 3-folds $Y$.
\begin{itemize}
\item The number of the corresponding entry in the Mori-Mukai list of rank 2
Fanos.
\item The anticanonical degree $-K_Y^3$ of the Fano $Y$.
\item The quadratic form of the Picard lattice $N$ presented with respect to a
basis $\bl$, $\bm$ that spans the nef cone (\ie the ample classes are exactly
the linear combinations of $\bl$ and $\bm$ where both coefficients are
positive); we do not know a general reason why the extremal rays should
generate all of $N$ and not just a finite index sublattice, but it does for all
entries in the list.
\item The absolute value $\Delta$ of the discriminant of the quadratic form
on $N$.
\item The anticanonical class $-K_Y$ in terms of the basis $\bl$, $\bm$.
\item
The element $\pi_!c_2(Z) = c_2(Y) + c_1(Y)^2 \in N^*$ (whose mod 24
reduction is $\bar c_2(Z) \in N^* \otimes \Z/24$ defined
in \eqref{eq:c2bar}) presented in the dual basis of $\bl$, $\bm$.
In other words, the entries of the row vector are the pairing of
$c_2(Z)$ with $\pi^*\bl$ and $\pi^*\bm$.
\item The third Betti number $b_3(Y)$ of the Fano.
\item The third Betti number $b_3(Z)$ of the resulting block.
\end{itemize}

\mmtable

The final four columns include data relevant for non-perpendicular matching,
about ample $A \in N$ such that $A^2$ is not too large compared with $\Delta$.
In view of Lemmas \ref{lem:orth} and \ref{lem:handarith}, non-perpendicular
matchings of a pair of rank 2 Fanos $Y_+, Y_-$ are only possible if
\begin{equation}
\label{eq:discr}
\log_2 \frac{\Delta_+}{A_+^2}+  \log_2\frac{\Delta_-}{A_-^2}\geq 0 .
\end{equation}
The largest ratio $\frac{\Delta}{A^2}$ occurs for \#18, which has $\Delta = 16$
and an ample $A$ with $A^2 = 10$.
Accordingly, we list all ample classes with $A^2 \leq 1.6\Delta$.
(In some examples there is more than one such class. We do not write out $A$
itself in terms of the basis $G, H$, but there is never any ambiguity.)
For each such $A$ we list the following data.
\begin{itemize}
\item The result of evaluating $\pi_! c_2(Z)$ on $A$
(equivalently: $c_2(Z) \pi^*A$).
\item $A^2$, the product of $A$ with itself in the Picard lattice.
\item $B^2$, where $B$ is a generator for the orthogonal complement to $A$
in the Picard lattice.
\item $h := \log_2 \frac{\Delta}{A^2}$.
\end{itemize}

\pagebreak[2]

\subsection{Annotated Mori-Mukai list of rank 2 Fano 3-folds}
\label{subsec:MM}

We now indicate how the data in Table~\ref{table:blocks} is assembled.
The anticanonical degree $-K_Y^3$ and $b_3(Y)$ are taken from
Iskovskih-Prokhorov \mbox{\cite[Table 12.3]{iskovskih99}}, and
$b_3(Z)$ is obtained from \eqref{eq:b3} as before.
For the computation of the basis of the nef cone, the quadratic form on the
Picard lattice $N$ and $\pi_! c_2(Z)$ we divide the list (except the last two
entries) into three groups.

No fewer than 27 of the 36 classes of Fanos can be described explicitly as
blow-ups in a smooth curve of a rank 1 Fano $Y_0$ of index $r \geq 2$
(so $\Pic Y_0$ is generated by ${-}\frac{1}{r}K_{Y_0}$).
Then one edge of the nef cone of $Y$
is clearly generated by $\bm := \pi^*({-}\frac{1}{r}K_{Y_0})$.

The hypothesis of the next lemma can be read as an elementary 
formulation of
``$C$~is cut out scheme-theoretically by sections of~$L$'', \ie
the tensor product $\cali_C \otimes \call$ of the ideal sheaf
of~$C$ and the sheaf of sections of $L$ being globally generated.

\begin{lem}
\label{lem:cutout}
Let $L$ be a line bundle on a closed complex manifold $Y_0$, and let
$\pi : Y \to Y_0$ be the blow-up of a smooth curve $C$ in $Y_0$.
Let $E \subset Y$ be the exceptional divisor of $\pi$.

Suppose that for every trivialising neighbourhood $U \subset Y_0$ of $L$,
the ideal $\mkern1.4mu\cali_C(U) := \{ f %
: f_{|C\cap U} \equiv 0\}$ in the ring $\calo_{Y_0}(U) := \{ \textrm{holomorphic functions } U \to \C\}$
is generated by coordinate representatives of 
global sections of $L$ that vanish identically on $C$.
Then $\pi^*L - E$ is basepoint-free.
\end{lem}

\begin{proof}
Note that if $D \in |L|$ contains $C$, then $\pi^*D - E$ is effective
(if $D$ is smooth then this is simply the proper transform) and belongs
to $|\pi^*L - E|$. Therefore if $|\pi^*L - E|$ has a base
point $p \in Y \setminus E$, then any global section of $L$ vanishes at the
corresponding point $\pi(p) \in Y_0$. The hypothesis forces the contradiction
$\pi(p) \in C$.

Recall that $E$ can be identified with $\PP(N_{C/Y_0})$, the projectivisation
of the normal bundle of $C$.
If $|\pi^*L - E|$ has a basepoint $p \in E$, then that corresponds to a
non-zero normal vector $v$ to $C$ at $\pi(p) \in C$ such that $v$ is tangent to
every element of $|L|$ that contains $C$. Then the derivative of every local
defining function of $C$ near $\pi(p)$ vanishes on $v$, contradicting that $C$
is smooth.
\end{proof}

Returning to the setting of a blow-up in a rank 1 Fano $Y_0$,
Lemma \ref{lem:cutout} implies that if $C$ is cut out by sections of
${-}\frac{n}{r}K_{Y_0}$ then $\bl = \pi^*({-}\frac{n}{r}K_{Y_0}) - E$ is nef,
so the cone spanned by $\bl$ and $\bm$ is certainly contained in the nef cone
of $Y$.

\begin{lem}
\label{lem:nef}
If the blow-up $Y$ of $Y_0$ in $C$ is Fano, and $n$ is the minimal integer such $C$
is cut out by sections of ${-}\frac{n}{r}K_{Y_0}$, then
the nef cone of $Y$ is exactly the cone spanned by 
$\bl = \pi^*({-}\frac{n}{r}K_{Y_0}) - E$ and $\bm$.
\end{lem}

\begin{proof}
This can be verified for each of the 27 rank 2 Fano 3-folds of this type in
the classification, aided \eg by the descriptions of
Coates-Corti-Galkin-Kasprzyk \mbox{\cite[\S18--\S53]{coates16}}
of each Fano 3-fold as the zero locus of a general section of a vector bundle
(often a complete intersection) of a ``key variety'', a GIT quotient of a vector space (often a toric variety) whose nef cone is evident.
\end{proof}

Note that $Y$ being Fano means that $-K_Y = \bl + (r-n)\bm$ is in the
interior of the nef cone, corresponding to $1 \leq n \leq r-1$ in all cases.
Indeed, according to \cite[Corollary 7.1.2(iii)]{iskov77}, the blow-up
of a rank 1 Fano 3-fold $Y_0$ in $C$ is Fano if and only if $1 \leq n \leq r-1$;
this too is regarded as a consequence of the classification, similarly to how
we do not know any uniform proof of Lemma~\ref{lem:nef}.

Staying in the context fo Lemma \ref{lem:nef}, the matrix representing the
quadratic form on $N$ with respect to the basis $\bl$, $\bm$
can be computed from
\begin{align*}
\bm^2 \; &= \; \frac{(-K_{Y_0})^3}{r^2} , \\
(n \bm {-} \bl).\bm \; &= \; \deg C , \\
(n \bm {-} \bl)^2 \; &= \; - \chi(C) .
\end{align*}
We can apply \eqref{eq:c2onH} to read off $\pi_!c_2(Z)(-K_Y)$
from the other data in the table. To control the other half
of $\pi_! c_2(Z)$, we apply Lemma \ref{lem:c2blowup} to $\pi_0 : Y \to Y_0$
to deduce that the product of $c_2(Y) + c_1(Y)^2$ with $\pi_0^*(-K_{Y_0})$
equals the product of $c_2(Y_0) + c_1(Y_0)^2$ with $-K_{Y_0}$.
We can thus apply \eqref{eq:c2onH} again to obtain
\begin{equation}
\label{eq:c2venient}
\pi_! c_2(Z) H = \frac{24-K_{Y_0}^3}{r}  
\end{equation}
(and the RHS is contained in Table \ref{table:species1}).

Another four entries in the list are divisors in $\PP^2 \times \PP^2$
(including $\PP^1 \times \PP^2$). Then we can take the basis $\bl, \bm$ of $N$
to correspond to the restrictions of the hyperplane bundles from the two
factors.
For a divisor $Y \in |a\bl+b\bm|$, we can readily compute the quadratic form
on $N$ from $-K_Y = (3-a)\bl + (3-b)\bm$. Further,
$c(Y) = (1+3\bl + 3\bl^2)(1+3\bm+3\bm^2)(1+[Y])^{-1}$ gives
\begin{align*}
c_2(Y) + c_1(Y)^2
&= 3\bl^2 + 9\bl\bm + 3\bl^2 - (3\bl+3\bm)[Y] + [Y]^2 + (3\bl + 3\bm- [Y])^2 \\
&= 12\bl^2 + 27\bl\bm + 12\bm^2 -(9\bl + 9\bm)[Y] + 2[Y]^2 .
\end{align*}
To evaluate the product with $\bl$ we compute
\begin{equation}
\label{eq:c2divisor}
(c_2(Y) + c_1(Y)^2)[Y]\bl = 3((2a-3)(b-1)(b-2) + 6)\bl^2\bm^2,
\end{equation}
and the product with $\bm$ is analogous. %

Another 3 cases are branched double covers over other rank 2 Fanos,
$\dcov : Y \to X$.
Then we can take $G$ and $H$ to be pull-backs of the previously identified
edges of the nef cone in $\Pic X$.
If the branch locus is a smooth divisor in $|2L|$, then
$TY \oplus \dcov^*(2L) \cong \dcov^*(TX \oplus L)$ implies
$c_1(Y) = \dcov^*(c_1(X) - [L])$ and
$c_2(Y) = \dcov^*(c_2(X) -c_1(X)[L] + 2[L]^2)$, so
\begin{equation}
\label{eq:c2double}
c_2(Y) + c_1(Y)^2 = \dcov^*\left(c_2(X) + c_1(X)^2 + 3[L]([L] - c_1(X))\right) .
\end{equation}

For each entry $Y$ in the Mori-Mukai list, we repeat below the description
from \cite[Table 12.3]{iskovskih99}. For blow-ups of rank 1 Fanos in a smooth
curve $C$ we only indicate in addition the smallest integer $n$ such that
$C$ is cut out by sections of ${-}\frac{n}{r}K_{Y_0}$. For the remaining
9 cases we provide some additional explanation.

\begin{enumerate}[label=\protect\#\arabic*,widest=99,itemsep=1mm]
\item \label{mm:v1}
Blow-up of $V_1$ (degree 1 del Pezzo 3-fold, degree 6 hypersurface in
$\bbp^4(3,2,1,1,1)$) in an elliptic curve that is the intersection of
two divisors in $|{-}\half K_{V_1}|$ (\ie the hyperplane class). $n = 1$.

This is the only rank 2 Fano where the linear system $|{-}K_Y|$ is not free;
the base locus is the pre-image $\PP^1$ over the basepoint of
$|{-}\half K_{V_1}|$. It is therefore the only case where
Construction \ref{constr:block} does not produce an
associated ``Fano-type'' building block $Z$ (though one could define a
building block by blowing up in several steps \cite[Proposition 5.9]{cym}).

\item \label{mm:24branch}
Double cover of $\PP^1 \times \PP^2$ branched over $(2,4)$ divisor.
$\bl$ and $\bm$ are the pull-backs of $\calo_{\PP^1}(1)$ and $\calo_{\PP^2}(1)$
respectively.
Use \eqref{eq:c2double} (and result from \ref{mm:p1p2}) to compute
\[ c_2(Y) + c_1(Y)^2 = 18\bl \bm + 12\bm^2 +3(\bl + 2\bm)(-\bl-\bm)
= 9\bl \bm + 6\bm^2 , \]
and hence $\pi_! c_2(Z) = \rvec{12}{18}$.

\item Blow-up of $V_2$ (degree 2 del Pezzo 3-fold, a double cover of $\PP^3$
branched over a quartic hypersurface, or equivalently a degree 4 hypersurface
in $\PP^3(2, 1, 1, 1)$) in an elliptic curve that is the intersection of two
hyperplane divisors. $n = 1$.

\item Blow-up of $\PP^3$ along the intersection of two cubic hypersurfaces.
$n = 3$.

\item Blow-up of $V_3$ (cubic hypersurface of $\PP^4$) along the intersection
of two hyperplane divisors (a plane cubic curve). $n = 1$.

\item
Generic members of the family are (2,2) divisor in $\PP^2 \times \PP^2$,
but some are double cover of a $(1,1)$ divisor $W \subset \PP^2 \times \PP^2$
branched over smooth divisor $B \in |{-}K_W|$.
$\bl$ and $\bm$ are pull-backs of $\calo(1)$ from the two
$\PP^2$ factors, and $\pi_! c_2(Z)$ can be computed from \eqref{eq:c2divisor}
(or from \eqref{eq:c2double} and the result in \ref{mm:11div}).

\item Blow-up of $Q$ (quadric hypersurface in $\PP^4$) in intersection of
two sections by quadrics. $n = 2$.

\item Double cover of $V_7$ ($\PP^3$ blown up in a point, \ref{mm:p3onepoint})
whose branch locus is a divisor $B \in |{-}K_{V_7}|$ (for generic members of the
intersection $B \cap E$ with the exceptional divisor of $V_7 \to \PP^3$ is
smooth, but for some members of the family $B$ is reduced but not smooth).
$\bl$ and $\bm$ are the pull-backs of the respective
classes on $V_7$.

By \eqref{eq:c2double} we have
$c_2(Y) + c_1(Y)^2 =
\dcov^*\left(c_2(V_7) + c_1(V_7)^2 - 3 (-\half K_{V_7})^2)\right)$.
We compute in \ref{mm:p3onepoint} that $c_2(V_7) + c_1(V_7)^2 = \rvec{18}{22}$
in terms of the basis for $V_7$, and we can read off from the Picard lattice
that $(-\half K_{V_7})^2 = \rvec{3}{4}$. Hence
$c_2(Y) + c_1(Y)^2 = \dcov^*\!\rvec{9}{10}$, which in terms of the basis for
$\Pic Y$ is written as $\rvec{18}{20}$.

\item Blow-up of $\PP^3$ in a curve $C$ of degree 7 and genus 5, which is
an intersection of a two-parameter family of cubic hypersurfaces. $n = 3$.

\item Blow-up of $V_4$ (complete intersection of two quadrics in $\PP^5$)
in an elliptic curve that is the intersection of two hyperplane sections.
$n = 1$.

\item Blow-up of $V_3$ (cubic hypersurface in $\PP^4$) along a line. $n = 1$.

\item Blow-up of $\PP^3$ along a curve of degree 6 and genus 3 which is
an intersection of cubic hypersurfaces. $n = 3$.

\item Blow-up of $Q$ (quadric hypersurface in $\PP^4$) along a curve of
degree 6 and genus 2. %
$n = 2$.

\item Blow-up of $V_5$ (section of Plücker-embedded Grassmannian
$Gr(2,5) \subset \PP^9$ by a subspace of codimension 3)
in an elliptic curve that is the intersection of two hyperplane sections.
$n = 1$.

\pagebreak[2]

\item Blow-up of $\PP^3$ along the intersection of a quadric $A$ and a cubic
$B$ (for generic members of the family $A$ is smooth, but for some members of the family $A$ is reduced but not smooth). $n = 3$.

\item Blow-up of $V_4$ (complete intersection of two quadrics in $\PP^5$)
in a conic. $n = 1$.

\item Blow-up of $Q$ along an elliptic curve of degree 5. $n = 2$.

\item Double cover of $\PP^1 \times \PP^2$ branched over $(2,2)$ divisor.
Compute $\pi_! c_2(Z) = \rvec{12}{18}$ using \eqref{eq:c2double} like in
\ref{mm:24branch}.

\item Blow-up of $V_4$ along a line. $n = 1$.

\item Blow-up of $V_5$ along a twisted cubic. $n = 1$.

\item Blow-up of $Q$ along a twisted quartic (a rational degree 4 curve,
isomorphic to the image of
$(s : t) \mapsto (s^4 : s^3t : s^2t^2 : st^3 : t^4)$). $n = 2$.

\item Blow-up of $V_5$ along a conic. $n = 1$.

\item Blow-up of $Q$ along an intersection of two divisors $A \in |\calo_Q(1)|$
and $B \in |\calo_Q(2)|$ ($A$ may be smooth or singular). $n = 2$.

\item A $(1,2)$ divisor in $\PP^2 \times \PP^2$.
Compute $\pi_! c_2(Z)$ by applying \eqref{eq:c2divisor} with $a = 1, b = 2$. 

\item Blow-up of $\PP^3$ along an elliptic curve that is the intersection of
two quadrics. $n = 2$.

\item Blow-up of $V_5$ along a line. $n = 1$.

\item Blow-up of $\PP^3$ along a twisted cubic. $n = 2$.

\item Blow-up of $\PP^3$ along a plane cubic (an elliptic curve). $n = 3$.

\item Blow-up of $Q$ along a conic (complete intersection of two hyperplane
sections). $n = 1$.

\item Blow-up of $\PP^3$ along a conic. $n = 2$.

\item Blow-up of $Q$ along a line. $n = 1$.

\item \label{mm:11div}
A $(1,1)$-divisor on $\PP^2 \times \PP^2$.
Compute $\pi_! c_2(Z)$ by applying \eqref{eq:c2divisor}
with $a = b = 1$. 

\item Blow-up of $\PP^3$ along a line. $n = 1$.

\item \label{mm:p1p2}
$Y = \PP^1 \times \PP^2$. $\bl = \calo_{\PP^1}(1)$ and $\bm = \calo_{\PP^2}(1)$.
Compute $\pi_! c_2(Z) = \rvec{12}{18}$ by applying \eqref{eq:c2divisor}
with $a = 1$, $b = 0$. 

\item \label{mm:p3onepoint} $\PP^3$ blown up in one point, or equivalently
$\PP(\calo \oplus \calo(1))$ over $\PP^2$. $\bl$ is the proper transform of
a plane passing through the blow-up point, and $\bm$ is
a plane not passing through the blow-up point. The product of
$c_2(Y) + c_1(Y)^2$ with $\bm$ is 22, just as it is for $\PP^3$.

\item $\PP(\calo \oplus \calo(2))$ over $\PP^2$. $\bl$ is the pull-back of
$\calo_{\PP^2}(1)$, while $\bm$ is the dual of the tautological bundle.
The intersection form on $H^2(Y)$ is given by $\bl^3 = 0$, $\bl^2\bm = 1$,
$\bl\bm^2 = 2$, $\bm^3 = 4$
(use that the section $\PP(\calo(2))$ is a divisor representing $\bm$).
Because $TY$ is stably isomorphic to
$f^*T\PP^2 \oplus \left(\bm \otimes f^*(\calo \oplus \calo(-2))\right)$,
where $f : Y \to \PP^2$ is the fibration, we find that
\[ c(Y) = (1 + 3\bl + 3\bl^2)(1+\bm)(1-2\bl+\bm), \]
so $c_1(Y) = \bl + 2\bm$ and $c_2(Y) = -3\bl^2 + 4\bl\bm +T^2$.
Hence $c_2(Y) + c_1(Y)^2 = -2\bl^2 +8\bl\bm +5\bm^2$, which has product
18 with $\bl$ and 34 with $\bm$. Thus $\pi_! c_2(Z) = \rvec{18}{34}$.
\end{enumerate}

\begin{rmk}
\label{rmk:nmarkedset}
In each case we have described a basis for the Picard lattice $N$,
which is tantamount to specifying an $N$-marking in the sense of
Definition \ref{def:nmarked}. On an elementary level, we could therefore
interpret each entry in the list as defining a set $\sff$ of $N$-marked Fano
3-folds.
\end{rmk}

\section{The matching problem}
\label{sec:match}

Combining the results described in \S \ref{subsec:tcs_def} and \S
\ref{subsec:block_from_fano}, we can produce twisted connected sum \gtmfd s
from matching pairs of Fano 3-folds. We will apply the methods developed in
\cite{g2m} to the problem of \emph{finding} matchings between Fanos of rank 1
and 2. In this section, we summarise the results from \cite[\S6]{g2m} on
finding matchings with a prescribed ``configuration'' of the Picard lattices of
a pair of Fano 3-folds,
reducing the problem %
to a combination of problems in lattice arithmetic and deformation theory.
The main result here---Proposition \ref{prop:match}---improves
on \cite[Proposition 6.18]{g2m} to deal more clearly with skew configurations.

\pagebreak[2]

\subsection{Configurations and matching}
\label{subsec:config}

Let $\kd_\pm \subset Y_\pm$ be smooth anticanonical divisors of a pair
of Fanos, and $\hkr : \kd_+ \to \kd_-$ a matching.
Let $\hdg_+$ be a \emph{marking} of $\kd_+$, \ie an isometry
$\hdg_+ : H^2(\kd_+) \to L$ where $L$ is a copy of the K3 lattice
(the unique unimodular lattice of signature (3,19)).
Then $\hdg_- := h_+ \circ \hkr^*$ is a marking of $\kd_-$. The images
of $H^2(Y_\pm) \subset H^2(\kd_\pm)$ under $\hdg_\pm$ are a pair of
primitive sublattices $N_\pm \subset L$, isometric to the Picard lattices.
This pair is well-defined up to the action of the isometry group $O(L)$,
and plays a crucial role.

\begin{defn}
\label{def:config}
Given a pair of Fanos with Picard lattices $N_+$ and $N_-$, call a pair
of primitive embeddings $N_+, N_- \into L$ a \emph{configuration}. Two such
pairs of embeddings are considered equivalent if they are related by the action
of $O(L)$.

We call a configuration \emph{orthogonal} if the reflections of $L(\bbr)$ in
$N_+$ and $N_-$ commute. If in addition $N_+ \cap N_-$ is trivial then we call
the configuration \emph{perpendicular}. If the configuration is not orthogonal
then we call it \emph{skew}.
\end{defn}

We saw in \S \ref{subsec:tcs_top} that the homeomorphism invariants of the
twisted connected sum $M$ resulting from the matching depend on the
configuration (\eg $H^2(M) = N_+ \cap N_-$), and in Theorem \ref{thm:ek_tcs}
that the generalised Eells-Kuiper invariant $\EK(M)$ does too.
We therefore ask:

\smallskip
{\narrower \noindent \em Given a pair $\sff_+$, $\sff_-$ of deformation
types of Fano 3-folds, which configurations of embeddings $N_\pm \subset L$
of their Picard lattices arise from some matching of elements of $\sff_+$
and~$\sff_-$?\par}

\smallskip \noindent
We see below that it is not too hard to answer this when one of the types 
has Picard rank 1, and we will be able to say quite a lot when both types 
have Picard rank 2. In general the question is quite difficult, but in any case
a first step in simplifying it is to rephrase it as a
problem of finding suitable triples of classes in $L(\bbr) := L \otimes \bbr$.
Recall that the \emph{period} of a marked K3 surface
$(\Sigma, \hdg)$ is an oriented two-plane $\Pi \subset L(\bbr)$, the image
under $h : H^2(\Sigma; \bbr) \to L(\bbr)$ of the real and
imaginary parts of classes in $H^{2,0}(\Sigma; \bbc)$.

\begin{lem}
\label{lem:triples}
Let $Y_\pm$ be a pair of Fano 3-folds, and let $N_\pm \subset L$ be the images
of primitive isometric embeddings of the respective Picard lattices.
Then the pair $(N_+, N_-)$ is the configuration
of some matching of $Y_+$ and $Y_-$ if and only if there exist
\begin{itemize}
\item an orthonormal triple $(k_+, k_-, k_0)$ of positive classes in~$L(\bbr)$,
\item anticanonical divisors $\kd_\pm \subset Y_\pm$,
\item markings $\hdg_\pm$ of $\kd_\pm$,
\end{itemize}
such that the oriented plane $\gen{k_\mp, \pm k_0}$ is the period of
$(\kd_\pm, \hdg_\pm)$, 
$\hdg^{-1}_\pm(k_\pm)$ is the restriction of a Kähler class on~$Y_\pm$,
and $N_\pm$ is the image of the composition
$H^2(Y_\pm) \to H^2(\kd_\pm) \to L$.
\end{lem}

\begin{proof}
Necessity is trivial, setting $k_\pm = h_\pm(\kclass_{\pm|\kd_\pm})$
for the Kähler classes $\kclass_\pm$ that appear in Definitions~ \ref{def:match},
and $k_0$ corresponding to a generator of $\Pi_+ \cap \hkr^* \Pi_-$,
all normalised to unit length. 
Sufficiency relies on the Torelli theorem,
\cf \mbox{\cite[Proposition 6.2]{g2m}}.
\end{proof}

To study how the matching problem depends on the choice of
configuration, let us first set up some notation for various lattices.
\begin{itemize}
\item $W  := N_+ + N_-$ (this need not be primitive in $L$),
\item $T_\pm \subset L$ the perpendicular of $N_\pm$,
\item $T := T_+ \cap T_-$, or equivalently the perpendicular of $W$,
\item $W_\pm := T_\mp \cap N_\pm$, and
\item $\Lambda_\pm \subset L$ the perpendicular to $T \oplus W_\mp$,
or equivalently the perpendicular to $W_\mp$ in the primitive overlattice
of $W$.
\end{itemize}

\begin{rmk}
\label{rmk:orth}
$N_\pm \subseteq \Lambda_\pm$, with equality if and only if $N_+$ and $N_-$
``intersect orthogonally'', \ie when
$W(\bbr) = W_+(\bbr) \oplus W_-(\bbr) \oplus (N_+(\bbr) \cap N_-(\bbr))$;
equivalently the configuration is orthogonal in the sense of
Definition \ref{def:config}.
\end{rmk}

\subsection{Necessary conditions}
\label{subsec:necc}

Note that in Lemma \ref{lem:triples} we must obviously
have $k_\pm \in N_\pm(\bbr)$.
On the other hand, $N_\pm$ is contained in the Picard group of the marked
K3 $(\kd_\pm, \hdg_\pm)$, which is the subgroup of $L$ orthogonal to
the period; the marked K3 is automatically \emph{$N_\pm$-polarised}.
Thus $k_\mp$ and $k_0$ must both lie in $T_\pm(\bbr)$.
Hence
\begin{equation}
\label{eq:periods}
k_\pm \in W_\pm(\bbr), \quad k_0 \in T(\bbr) .
\end{equation}
Now we come to the heart of how the difficulty of the matching problem depends
on the configuration one tries to achieve: \eqref{eq:periods} implies
that the period $\gen{k_\mp, \pm k_0}$ is orthogonal to all of $\Lambda_\pm$,
so %
\emph{the marked K3 divisors used in a matching with the given configuration
are forced to be $\Lambda_\pm$-polarised}.

The significance is that
the Picard group of a generic K3 divisor in a generic member of a deformation
type $\sff_\pm$ of Fano 3-folds will be precisely the Picard lattice $N_\pm$ of
that type. To find matchings for a configuration where $\Lambda_\pm$ is 
strictly bigger than $N_\pm$, we therefore require \emph{non-generic} K3
divisors in members of $\sff_\pm$ (the
moduli space of $\Lambda_\pm$-polarised marked K3 surfaces forms a subspace of
the $N_\pm$-polarised K3s, whose codimension is $\rk \Lambda_\pm - \rk N_\pm$).

For configurations where $\Lambda_\pm = N_\pm$, we deduce in \S\ref{sec:orth}
the existence
of matchings between some elements of $\sff_+$ and $\sff_-$ from a general fact
(due to Beauville \cite{beauville04}) that a generic $N_\pm$-polarised K3
appears as an anticanonical divisor in some member of $\sff_\pm$.
In view of Remark \ref{rmk:orth}, this comparatively easy case corresponds to
orthogonal configurations.
To apply a similar argument for skew configurations (where
$\Lambda_\pm \supset N_\pm$), we first need to show for those specific
$\Lambda_\pm$ that generic $\Lambda_\pm$-polarised K3s appear as anticanonical
divisors in members of $\sff_\pm$. Even when it is true, it is something that
we can so far only verify case by case.
We refer to this process as `handcrafting'.

\begin{rmk}
\label{rmk:arith}
Before moving on to existence results for matchings with a prescribed
configuration, let us point out some necessary conditions.
\begin{enumerate}
\item Since $W(\bbr)$ contains a two-dimensional positive-definite subspace
(spanned by the orthogonal classes $k_+$ and $k_-$),
while its orthogonal complement in the signature $(3,19)$ lattice $L$
contains a class $k_0$ with $k_0^2 > 0$, the
quadratic form on $W$ must be non-degenerate of signature ${(2, \rk(W)-2)}$.
\item \label{it:ample}
$W_\pm \subset N_\pm$ must contain some ample classes of $Y_\pm$.
\item Since $\Lambda_+ \cap \Lambda_- \subset \Pic \kd_\pm$ and is orthogonal
to an ample class of $\kd_\pm$, it cannot contain any $(-2)$-classes.
\end{enumerate}
\end{rmk}

\begin{rmk}
\label{rmk:trichotomy}
In particular, \ref{it:ample} implies that any matching involving a Fano
with Picard rank $\rk N= 1$ must be perpendicular.
Moreover, for a configuration of lattices $N_+$ and $N_-$ where
at least one has rank 2, if the intersection $N_+ \cap N_-$ is non-trivial
then \ref{it:ample} forces the configuration to be orthogonal
in the sense of Definition \ref{def:config}. %
For configurations of Picard lattices of Fanos of rank $\leq 2$ that satisfy
the necessary conditions to be realised by a matching, we therefore have the
following trichotomy:
\begin{itemize}
\item Perpendicular configurations, \ie every element of $N_+$ is orthogonal to
every element of $N_-$.
\item Orthogonal configurations with non-trivial intersection. Then
$N_+ \cap N_-$ must have rank 1.
\item Skew configurations. Then $N_+ \cap N_-$ must be trivial, but $N_+$ is
not perpendicular to $N_-$ (the maps $N_\pm \to N_\mp^*$ must have rank 1).
\end{itemize}
In \S\ref{sec:orth} and \S\ref{sec:handcraft}, we will consider these cases in
turn.
\end{rmk}

\subsection{Sufficient conditions}

In order to describe the `genericity properties' we require for anti\-canonical
K3 divisors in families of Fano 3-folds, we recall some further
terminology. The \emph{period domain} is the space of oriented positive-definite
2-planes in $L(\bbr)$. It can be identified with
${\{ \Pi \in \bbp(L(\bbc)) : \Pi^2 = 0, \; \Pi \, \overline\Pi > 0\}}$
in order to exhibit a natural complex structure.
Given $\Lambda \subset L$, the period domain of $\Lambda$-polarised K3
surfaces is $D_\Lambda :=
\{ \Pi \in \bbp(\Lambda^\perp(\bbc)) : \Pi^2 = 0,\; \Pi \, \overline\Pi > 0\}$.

\begin{defn}
\label{def:nmarked}
Given a non-degenerate lattice $N$, an \emph{$N$-marking} of a closed
3-fold $Y$ is a surjective homomorphism
$i_Y : H^2(Y) \to N$ that is isometric for %
the anticanonical form of Definition \ref{def:acform}.
\end{defn}

We avoid calling $i_Y$ an ``$N$-polarisation'' since we do not impose any
conditions on ample classes.
If $Y$ is Fano
then the Picard lattice is non-degenerate so $i_Y$ is simply an
isometry. %

\begin{defn}
\label{def:generic}
Let $N \subseteq \Lambda \subset L$ be primitive non-degenerate sublattices
of $L$, and
let $\Amp_\sff$ be a non-empty open subcone of the positive cone in~$N(\bbr)$.
We say that a set $\sff$ of $N$-marked 3-folds is
\emph{$(\Lambda, \Amp_\sff)$-generic} if there is
$U_\sff \subseteq D_\Lambda$ with complement a countable 
union of complex analytic submanifolds of positive codimension with the
property that:
for any $\Pi \in U_\sff$ and $k \in \Amp_\sff$ there is $Y \in \sff$, a smooth
anticanonical divisor $\kd \subset Y$ and a marking $\hdg : H^2(\kd) \to L$
such that $\Pi$ is the period of $(\kd, \hdg)$,
the composition $H^2(Y) \to H^2(\kd) \to L$ equals the marking~$i_Y$,
and $\hdg^{-1}(k)$
is the image of the restriction to $\kd$ of a Kähler class on~$Y$.
\end{defn}

To be able to prove that a set $\sff$ of Fano 3-folds satisfies the
definition we typically take $\sff$ to be a deformation type,
but to make sense of the definition we do not need to remember any additional
structure on $\sff$ (\cf Remark \ref{rmk:nmarkedset}).

Meanwhile, when applying the next proposition we typically want all elements of
the sets $\sff_\pm$ to be Fano 3-folds (or building blocks) that are
topologically the same,
so that we have some control over the topology of the \gtmfd s resulting from
the matchings produced; essentially this means that all elements of $\sff_\pm$
should belong to the same deformation type.

\begin{prop}
\label{prop:match}
Consider a configuration of primitive non-degenerate sublattices
$N_+, N_- \subset L$, and let $\sff_\pm$ be a pair of sets of $N_\pm$-marked
3-folds.
Define $W$, $W_\pm$ and $\Lambda_\pm$ as above. Suppose
that there exist non-empty open cones $\Amp_{\sff_\pm} \subseteq N_\pm(\bbr)^+$
such that
\begin{enumerate}
\item the sets $\sff_\pm$ are $(\Lambda_\pm, \Amp_{\sff_\pm})$-generic,
\item \label{it:amp} $W_\pm \cap \Amp_{\sff_\pm} \neq \emptyset$.
\end{enumerate}
Then there is an open dense subcone $\mathcal{W} \subseteq
(W_+(\R) \cap \Amp_{\sff_+}) \times (W_-(\R) \cap \Amp_{\sff_-})$
such that for every $(k_+, k_-) \in \mathcal{W}$ with $k_+^2 = k_-^2$
there exist $Y_\pm \in \sff_\pm$, anticanonical K3 divisors
$\kd_\pm \subset Y_\pm$
and Kähler classes $\kclass_\pm \in H^2(Y_\pm)$ such that
$\kclass_{\pm|\kd_\pm} = k_\pm$, with 
a matching $\hkr : \kd_+ \to \kd_-$ of $(Y_+, \kd_+, \kclass_+)$
and $(Y_-, \kd_-, \kclass_-)$ whose
configuration is the given pair of embeddings $N_\pm \subset L$.
\end{prop}

\begin{proof}
The argument is essentially the same as \cite[Proposition 6.18]{g2m},
even though the conclusion stated here is slightly stronger.

Let $T = W^\perp$ as before. Denote the ranks of $W$ and $W_\pm$ by $r$ and
$r_\pm$. Then $W_\pm(\bbr)$ and $T(\bbr)$ are real vector spaces
of signature $(1, \, r_\pm - 1)$ and $(1, \, 21 -r)$ respectively

In view of Lemma \ref{lem:triples} and \eqref{eq:periods}, matchings correspond
to certain triples of classes $(k_+, k_-, k_0)$ such that $k_\pm$ and $k_0$
belong to the positive cones $W_\pm(\bbr)^+$ and $T(\bbr)^+$ respectively.
Consider therefore the real manifold 
\[D= \PP\bigl(W_+(\R)^+\bigr)\times \PP\bigl(W_-(\R)^+\bigr)
\times \PP \bigl(T(\R)^+\bigr)\, .
\]
Below, we need the open subset
$\cala = \cala_+ \times \cala_- \times \PP\bigl(T(\R)^+\bigr)$,
where $\cala_\pm := \PP(W_\pm(\R) \cap \Amp_{\sff_\pm})$ is non-empty
by hypothesis \ref{it:amp}.
We have two Griffiths period domains
 \[  D_{\Lambda_\pm} =
\{\textrm{positive-definite planes }\Pi \subset \Lambda_\pm^\perp (\R) \} , \]
and projections
\[
\text{pr}_\pm \colon D \to D_{\Lambda_\pm}, \;
(\ell_+, \ell_-,\ell) \mapsto \gen{\ell_\mp, \pm \ell} .
\]
As stated before Definition \ref{def:generic}, $D_{\Lambda_\pm}$ can be
regarded as an open subset of $\PP(C_\pm)$, where $C_\pm$ is the null cone
in $\Lambda_\pm^\perp \otimes \C$;
if $\alpha, \beta$ is an oriented orthonormal
basis of $\Pi \in D_{\Lambda_\pm}$ then
$\Pi \mapsto \gen{\alpha+i\beta} \in \PP(C_\pm)$.
Given a choice $\alpha$ and~$\beta$, we can identify $T_\Pi D_{\Lambda_\pm}$
with pairs $(u,v)$ of vectors in the orthogonal complement of $\Pi$
in~$\Lambda_\pm^\perp(\R)$. Then the complex structure on $T_\Pi D_{N_\pm}$ is
given by $J : (u, v) \mapsto (-v, u)$.

Observe that the real analytic embedded submanifold
$\PP\bigl(W_\mp(\R)^+\bigr)\times \PP \bigl(T(\R)^+\bigr) \into
D_{\Lambda_\pm}$ is totally real: for $w \in W_\mp$ and $t \in T(\R)$,
the tangent space $\mathcal{T}$
to $\PP\bigl(W_\mp(\R)^+\bigr)\times \PP \bigl(T(\R)^+\bigr)$
at $\Pi = \gen{w,t}$ corresponds to $(u,v)$ such
that $u \in w^\perp \subseteq W_\mp(\R)$ and $v \in t^\perp \subseteq T(\R)$,
so $J(\mathcal{T})$ is transverse to $\mathcal{T}$.

Crucially, this totally real submanifold has \emph{maximal dimension}:
\[ \dim_\R \PP \bigl(W_\mp(\R)^+\bigr)\times \PP \bigl(T(\R)^+\bigr) =
(r_\mp-1)  + \bigl(22 - r - 1\bigr) = 20 - r + r_\mp =  
20-\rk \Lambda_\pm = \dim_\C D_{\Lambda_\pm}  \]
Consequently, its intersection with any positive-codimensional complex analytic
submanifold of $D_{\Lambda_\pm}$ is a positive-codimensional real analytic
submanifold of $\PP\bigl(W_\mp(\R)^+\bigr)\times \PP \bigl(T(\R)^+\bigr)$.
Hence the pre-image in
$\PP\bigl(W_\mp(\R)^+\bigr)\times \PP \bigl(T(\R)^+\bigr)$ of 
the subset $U_{\fbb_\pm} \subset D_{\Lambda_\pm}$ from
Definition \ref{def:generic} is open and dense.
Because $\text{pr}_\pm$ is a projection of a product manifold onto a factor
the same is true for $\text{pr}_\pm^{-1}(U_{\fbb_\pm}) \subset D$.
In turn,
\[ \bigl(\cala_+ \times \cala_- \times \PP(T(\R)^+)\bigr) \cap
\text{pr}_+^{-1}(U_{\fbb_+}) \cap \text{pr}_-^{-1}(U_{\fbb_-}) \]
is open and dense in $\cala_+ \times \cala_- \times \PP(T(\R)^+)$, and hence
the image $\mathcal{W}'$ of this subset under projection to
$\cala_+ \times \cala_-$ is open and dense in $\cala_+ \times \cala_-$.

If we let $\mathcal{W} = \{(k_+, k_-) \in 
(W_+(\R) \cap \Amp_{\sff_+}) \times (W_-(\R) \cap \Amp_{\sff_-}) :
([k_+], [k_-]) \in \mathcal{W}'\}$, then for every $(k_+, k_-) \in \mathcal{W}$
such that $k_+^2 = k_-^2$ there is a $k_0 \in T(\R)^+$ such that
Lemma \ref{lem:triples} applies to~$(k_+, k_-, k_0)$.
\end{proof}

\newcommand{\perptable}{
\begin{table}[t]
{\small
\setlength{\tabcolsep}{2pt}
\begin{tabular}[t]{C{24pt}C{20pt}C{20pt}C{16pt}@{\hspace{6pt}}*{4}{C{16pt}}}
\toprule
\multirow{2}{*}{$b_3$} & \multirow{2}{*}{$\#$} &
\multicolumn{6}{c}{$d$} \\ \cmidrule(l){3-8}
 &  & 2 & 4 & 6 & 8 & 12 & 24 \\ \midrule

 71  &  15  &   8  &   5  &   1  &   1  &   0  &   0  \\
 73  &  20  &  16  &   3  &   1  &   0  &   0  &   0  \\
 75  &  40  &  35  &   3  &   1  &   0  &   1  &   0  \\
 77  &  39  &  36  &   0  &   2  &   0  &   1  &   0  \\
 79  &  73  &  65  &   3  &   4  &   0  &   1  &   0  \\
 81  &  60  &  55  &   1  &   4  &   0  &   0  &   0  \\
 83  &  77  &  67  &   3  &   6  &   1  &   0  &   0  \\
 85  &  63  &  49  &   8  &   4  &   1  &   1  &   0  \\
 87  &  77  &  69  &   4  &   3  &   0  &   1  &   0  \\
 89  &  57  &  49  &   5  &   2  &   1  &   0  &   0  \\
 91  &  55  &  51  &   2  &   2  &   0  &   0  &   0  \\
 93  &  49  &  45  &   4  &   0  &   0  &   0  &   0  \\
 95  &  52  &  49  &   0  &   2  &   0  &   0  &   1  \\
 97  &  51  &  44  &   3  &   3  &   1  &   0  &   0  \\
 99  &  54  &  42  &   7  &   3  &   2  &   0  &   0  \\
101  &  33  &  26  &   1  &   5  &   1  &   0  &   0  \\
103  &  66  &  55  &   3  &   7  &   1  &   0  &   0  \\
105  &  41  &  38  &   0  &   3  &   0  &   0  &   0  \\
107  &  48  &  43  &   1  &   3  &   1  &   0  &   0  \\
109  &  34  &  31  &   0  &   3  &   0  &   0  &   0  \\
111  &  43  &  40  &   0  &   2  &   1  &   0  &   0  \\
113  &  40  &  34  &   5  &   0  &   1  &   0  &   0  \\
115  &  32  &  30  &   1  &   1  &   0  &   0  &   0  \\
117  &  30  &  28  &   1  &   0  &   1  &   0  &   0  \\
119  &  31  &  28  &   0  &   3  &   0  &   0  &   0  \\
121  &  29  &  26  &   1  &   2  &   0  &   0  &   0  \\
123  &  17  &  16  &   0  &   1  &   0  &   0  &   0  \\
125  &  11  &  10  &   1  &   0  &   0  &   0  &   0  \\
127  &  20  &  14  &   3  &   2  &   1  &   0  &   0  \\
129  &  11  &  10  &   0  &   1  &   0  &   0  &   0  \\
131  &   9  &   8  &   1  &   0  &   0  &   0  &   0  \\
\bottomrule
\end{tabular}
\hspace{1cm}
\begin{tabular}[t]{C{24pt}C{20pt}C{20pt}C{16pt}@{\hspace{6pt}}*{4}{C{16pt}}}
\toprule
\multirow{2}{*}{$b_3$} & \multirow{2}{*}{$\#$} &
\multicolumn{6}{c}{$d$} \\ \cmidrule(l){3-8}
 &  & 2 & 4 & 6 & 8 & 12 & 24 \\ \midrule
133  &   3  &   3  &   0  &   0  &   0  &   0  &   0  \\
135  &  10  &   9  &   0  &   1  &   0  &   0  &   0  \\
137  &  12  &  12  &   0  &   0  &   0  &   0  &   0  \\
139  &   4  &   4  &   0  &   0  &   0  &   0  &   0  \\
141  &   2  &   1  &   1  &   0  &   0  &   0  &   0  \\
143  &   3  &   3  &   0  &   0  &   0  &   0  &   0  \\
145  &   7  &   7  &   0  &   0  &   0  &   0  &   0  \\
147  &   2  &   2  &   0  &   0  &   0  &   0  &   0  \\
151  &   1  &   1  &   0  &   0  &   0  &   0  &   0  \\
153  &   2  &   2  &   0  &   0  &   0  &   0  &   0  \\
155  &   8  &   7  &   1  &   0  &   0  &   0  &   0  \\
157  &   4  &   4  &   0  &   0  &   0  &   0  &   0  \\
159  &   6  &   6  &   0  &   0  &   0  &   0  &   0  \\
161  &   3  &   3  &   0  &   0  &   0  &   0  &   0  \\
163  &   8  &   8  &   0  &   0  &   0  &   0  &   0  \\
165  &   2  &   2  &   0  &   0  &   0  &   0  &   0  \\
167  &   3  &   3  &   0  &   0  &   0  &   0  &   0  \\
169  &   3  &   3  &   0  &   0  &   0  &   0  &   0  \\
171  &   1  &   1  &   0  &   0  &   0  &   0  &   0  \\
173  &   2  &   2  &   0  &   0  &   0  &   0  &   0  \\
175  &   1  &   1  &   0  &   0  &   0  &   0  &   0  \\
179  &   4  &   4  &   0  &   0  &   0  &   0  &   0  \\
181  &   1  &   1  &   0  &   0  &   0  &   0  &   0  \\
183  &   1  &   1  &   0  &   0  &   0  &   0  &   0  \\
187  &   3  &   3  &   0  &   0  &   0  &   0  &   0  \\
189  &   1  &   1  &   0  &   0  &   0  &   0  &   0  \\
195  &   1  &   1  &   0  &   0  &   0  &   0  &   0  \\
197  &   2  &   2  &   0  &   0  &   0  &   0  &   0  \\
239  &   1  &   1  &   0  &   0  &   0  &   0  &   0  \\[0.5em]
\midrule
Total& 1378 & 1215 &  71  &  72  &  14  &   5  &   1  \\[1.6pt]
\bottomrule
\end{tabular}
}
\bigskip
\caption{Twisted connected sums from perpendicular matching of Fanos with Picard rank 1 or 2}
\label{table:perp}
\vspace{-4mm plus 20mm}
\end{table}
}

\section{Orthogonal matching}
\label{sec:orth}

We now consider the problem of finding matchings of Fano 3-folds of Picard rank
1 or 2, with prescribed configuration that is orthogonal in the sense of
Definition \ref{def:config}.
As pointed out in Remark~\ref{rmk:orth}, this corresponds to
the Picard lattices $N_\pm$ being equal to the lattices $\Lambda_\pm$ that
are used in the hypothesis of Proposition \ref{prop:match}.
Therefore the following genericity result is enough to let us apply
Proposition \ref{prop:match} for these configurations.

\begin{prop}[{\cite[Proposition 6.9]{cym}, based on Beauville
\cite{beauville04}}]
\label{prop:generic}
Let $\sff$ be a deformation type of Fano 3-folds, and embed its Picard lattice
$N$ primitively in $L$. Then there exists an open cone
$\Amp_\sff \subset N(\bbr)$ such that $\sff$ is $(N, \Amp_\sff)$-generic.
\end{prop}

In the trichotomy of Remark \ref{rmk:trichotomy}, orthogonal configurations
encompass the cases of perpendicular configurations and orthogonal
configurations with non-trivial intersection.
Now we study in turn the twisted connected sum \gtmfd s that result from
matchings of these kinds, using Fano 3-folds of Picard rank 1 or 2. 

\subsection{Perpendicular matching}

The simplest way to find a matching between elements of two deformation types
$\sff_\pm$ of Fano 3-folds is to consider perpendicular configurations,
\ie where the images of $N_+$ and $N_-$ in $L$ are perpendicular to each other.
One reason is that we do not need any genericity results beyond
Proposition \ref{prop:generic}, but a further reason is arithmetic:
for nearly all pairs $\sff_\pm$ one has
$\rk N_+ + \rk N_- \leq 11$, in which case Nikulin
\mbox{\cite[Theorem 1.12.4]{nikulin79}} guarantees that there does in fact
exist a primitive embedding of the perpendicular direct sum $N_+ \perp N_-$
into $L$.

\perptable

In particular, we can apply this to find perpendicular matchings among the 53
types of Fanos of Picard rank 1 and 2. However, we ignore \#1 in the list
of rank 2 Fanos, since that does not have an associated Fano-type building
block. Given one of the 1378 unordered pairs $\sff_+$, $\sff_-$ among the
other 52 deformation types, we carry out the following procedure:
\begin{itemize}
\item Apply \cite[Theorem 1.12.4]{nikulin79} to find a primitive embedding
into $L$ of the perpendicular direct sum $N_+ \perp N_-$.
\item Apply Proposition \ref{prop:match} to find a matching between some
$Y_\pm \in \sff_\pm$ with the given configuration $N_+, N_- \subset L$.
\item
Apply Proposition \ref{prop:block} to produce a pair of Fano-type building
blocks with a perpendicular matching.
\item Apply Theorem \ref{thm:g2glue} to construct a twisted connected
sum \gtmfd{} $M$.
\end{itemize}
Now Theorem \ref{thm:tcs_H} shows that $M$ is 2-connected, with $H^4(M)$ torsion-free, and
\[ b_3(M) = b_3(Z_+) + b_3(Z_-) + 23. \]
Proposition \ref{prop:p1y} implies that $d(M)$, the greatest divisor of $p_M$,
equals the greatest common divisor of $\pi_! c_2(Z_+)$ and $\pi_! c_2(Z_-)$,
while Corollary \ref{cor:perp0} shows that $\EK(M) = 0$.
Thus all the classifying diffeo\-morphism invariants of $M$ can be determined
from the data in Tables \ref{table:species1} and \ref{table:blocks}.

Table \ref{table:perp} lists the invariants of the 1378 twisted connected
sums obtained this way. A total of 131 different 2-connected manifolds are
realised, with 60 different values of $b_3(M)$.  For comparison, twisted
connected sums involving only rank 1 Fanos realise 82 different manifolds, and
46 different values of $b_3(M)$ \cite[Table 3]{g2m}.

\pagebreak[1]

\begin{rmk}
According to Theorem 1.7 and Corollary 1.13 of \cite{nu}, the torsion-free
\gtstr s of diffeomorphic 2-connected twisted connected sums with $d$ not
divisible by 3 are automatically homotopic (if one chooses the
diffeomorphism correctly). While Table \ref{table:perp} shows that there are
also numerous instances of diffeomorphic twisted connected sums with $d = 6$
(\eg $(b_3, d) = (103,6)$ is realised by 7 different twisted connected sums),
Wallis \cite[Corollary 6.5.2]{wallis:thesis} builds on the methods of this
paper to compute the \gtstr s' $\xi$-invariant introduced
in \cite[Definition 6.8]{nu}, showing that the \gtstr s are homotopic for these
examples too.
However, \cite[Table 8.1]{wallis:thesis} exhibits other examples (involving
skew matchings) where the \gtstr s of twisted connected sums can be
distinguished by the $\xi$-invariant.
\end{rmk}

\begin{rmk}
\label{rmk:circle_bundles}
Recall from Wall's classification of closed simply-connected
spin 6-manifolds with torsion-free cohomology \cite[Theorem 5]{wall66v} that
for $d$ divisible by 12, there is such a 6-manifold $X$ with the cohomology
ring of $S^2 \times S^4$ and greatest divisor of $p_X$ equal to $d$.

The total space $M$ of the circle bundle over $X \# (S^3 \times S^3)^{\# k}$
with Euler class equal to the generator of $H^2$ is then a closed 2-connected
7-manifold with $b_3(M) = 2k+1$, torsion-free $H^4(M)$ and $d(M) = d$ (and this
is the only way to obtain such total spaces). One can use the corresponding
closed disc bundle as a coboundary for $M$ to compute that $\EK(M) = 0$.
(The more general problem of deciding which 2-connected 7-manifolds
can be realised as the total spaces of circle bundles 
is analysed by Jiang \cite{jiang14}.)

Thus the classification Theorem \ref{thm:2c7m} implies that the 6 entries
in Table \ref{table:perp} that have $d = 12$ or $24$ are diffeomorphic to
total spaces of circle bundles (in contrast to the result of
Baraglia \mbox{\cite[Proposition 6.2.1]{baraglia09}} that closed
\gtmfd s cannot be smoothly fibred by 4-manifolds), while the last entry of
Table \ref{table:mainex} is a topological $S^1$-bundle but not a smooth one.
\end{rmk}

\subsection{Orthogonal matching with non-trivial intersection}

\renewcommand{\thefootnote}{$\dagger$}

\newcommand{\orthtable}
{
\setlength{\belowcaptionskip}{4pt}
\begin{table}[tb]
\[
\setlength{\arraycolsep}{8pt}
\begin{array}[b]{rr@{\hspace{12pt}}rrrr@{\hspace{18pt}}r} \toprule
\rbox{\#_+}{10pt} & \rbox{\#_-}{10pt} & \rbox{B^2}{2pt}
& \rbox{A_+^2}{6pt} & \rbox{A_-^2}{6pt}
& \rbox{b_3(M)}{12pt} & \rbox{d(M)}{14pt} \\
\midrule
  6  &   6  &  -4 &  12   &  12   &   86   &  12 \\
  6  &  12  &  -4 &  12   &  20   &   82   &   4 \\
  6  &  21  &  -4 &  12   &  28   &   84   &   4 \\
  6  &  32  &  -4 &  12   &  12   &  104   &  12 \\
 12  &  12  &  -4 &  20   &  20   &   78   &   4 \\
 12  &  21  &  -4 &  20   &  28   &   80   &   4 \\
 12  &  32  &  -4 &  20   &  12   &  100   &   4 \\
 21  &  21  &  -4 &  28   &  28   &   82   &   4 \\
 21  &  32  &  -4 &  28   &  12   &  102   &   4 \\
 32  &  32  &  -4 &  12   &  12   &  122   &  12 \\
  2  &  24  &  -6 &   6   &  14   &  102   &   6 \\
 18  &  24  &  -6 &  24   &  14   &   84   &  12 \\
\footnotemark \;\;
  5  &  25  & -12 &  12   &  12   &   84   &   2 \\
 10  &  10  & -16 &  16   &  16   &   70   &   8 \\
 14  &  22  & -20 &  20   &  30   &   78   &   2 \\
 14  &  25  & -20 &  20   &  20   &   82   &   2 \\
 13  &  14  & -30 &  20   &  30   &   72   &   4 \\
 14  &  18  & -40 &  40   &  10   &   76   &   2 \\
 18  &  34  & -72 &  18   &   8   &  108   &   6 \\
\bottomrule
\end{array}
\]
\caption{Twisted connected sums $M$ from matchings of rank 2 blocks with
non-trivial intersection $N_+ \cap N_-$}
\label{table:orth}
\vspace{-4mm plus 20mm}
\end{table}
}

Next we consider matchings with configurations such that $N_+ \cap N_-$
is non-trivial. Then both Fanos must have Picard rank $\geq 2$. If we restrict
attention to the case when both Fanos have Picard rank precisely 2, then
as pointed out in Remark \ref{rmk:trichotomy} the only configurations with
$N_+ \cap N_-$ non-trivial for which we can possibly find a matching are
the ones that are orthogonal, in the sense of Definition \ref{def:config}.
Such configurations have $\Lambda_\pm = N_\pm$, so
to apply Proposition \ref{prop:match} we essentially do not need any genericity
result beyond Proposition \ref{prop:generic} that we applied to find
perpendicular matchings---the only extra data we need is to actually determine
the cone $\Amp_\sff$, but for rank 2 Fano 3-folds we have done 
that in \S\ref{subsec:MM}.

Compared with the perpendicular matching problem, the difficulty of finding
matchings of rank 2 blocks with non-trivial intersection $N_+ \cap N_-$
is therefore one of lattice-arithmetic: there must exist some integral lattice
$W$ of rank 3, containing $N_+$ and $N_-$, such that the orthogonal complement
of $W_\pm \subset N_\pm$ of $N_\mp$ is non-trivial, and contains a class
$A_\pm \in \Amp_{\sff_\pm}$.

\begin{lem}
\label{lem:orth}
Let $N_\pm$ be integral lattices of rank 2 with signature (1,1),
and let $A_\pm \in N_\pm$. Let $-\Delta_\pm$ be the discriminant of $N_\pm$, and
let $B_\pm$ be a generator of the orthogonal complement.
Then there exists a rank 3 integral
lattice $W$ with primitive isometric embeddings $N_\pm \into W$ such that
$A_\pm \perp N_\mp$ if and only if
$B_+^2 = B_-^2$ and $\Delta_+\Delta_- = k^2 A_+^2A_-^2$ for some integer $k$.
Then $k$ is the generator of the image of the product $N_+ \times N_- \to \Z$
in $W$.
\end{lem}

\pagebreak[2]

\begin{proof}
If such a $W$ exists, then the images of $B_+$ and $B_-$ are both primitive
vectors, and both perpendicular to both $A_+$ and $A_-$. Thus, up to sign,
$B_+$ and $B_-$ must have the same image $B \in W$,
and hence $B_+^2 = B^2 = B_-^2$.

Note that for a pair $(v_+, v_-) \in N_+ \times N_-$,
the product of their images in $W$ is
\begin{equation}
\label{eq:extend}
v_+. v_- = \frac{(v_+.B_+)(v_-.B_-)}{B^2} .
\end{equation}
In particular, if we let $a_\pm$ be the positive generator of the image of
$N_\pm \to \bbz, \; v \mapsto v.B_\pm$, then $a_+a_- = -kB^2$.
Conversely if $a_+a_-$ is divisible by $B^2$, then we can use \eqref{eq:extend}
(together with the given forms on $N_+$ and $N_-$) to define the desired
integral quadratic form
on $W := (N_+ \oplus N_-)/\gen{B_+ - B_-}$.

Now observe that the index of the sublattice $\gen{A_+, B_+} \subseteq N_\pm$
is $\frac{B^2}{a_\pm}$. Letting $-\Delta_\pm$ be the discriminant of $N_\pm$,
we must therefore have
$A_+^2 B^2 = (-\Delta_\pm)\left(\frac{B^2}{a_\pm}\right)^2$. Hence
\[ (a_+ a_-)^2 = \frac{\Delta_+ \Delta_- (B^2)^2}{A_+^2 A_-^2} , \]
and $a_+a_-$ is divisible by $B^2$ if and only if
$\frac{\Delta_+\Delta_-}{A_+^2 A_-^2}$ is a perfect square.
\end{proof}

\pagebreak[3]

Next we summarise the relevant topological calculations.

\begin{lem}
\label{lem:orth_top}
Let $(Z_\pm, \kd_\pm)$ be a pair of building blocks whose polarising lattices
$N_\pm$ have rank~2, such that the kernel of $H^2(Z_\pm) \to H^2(\kd_\pm)$
is generated by $[\kd_\pm]$.
Let $\hkr : \kd_+ \to \kd_-$ be a matching whose configuration
$N_+, N_- \subset L$ has $N_+ \cap N_-$ of rank 1, and let
$M$ be the resulting twisted connected sum.

Let $W := N_+ + N_-$, and let $A_\pm \in N_\pm$ be a primitive vector
spanning the orthogonal complement of $N_\mp$ in $N_\pm$. Then
\begin{enumerate}
\item $H^2(M) \cong \Z$;
\item $b_3(M) = b_3(Z_+) + b_3(Z_-) + 22$;
\item $\Tor H^3(M) \cong \Tor L/W$;
\item \label{it:orth_tor}
$\Tor H^4(M) \cong (\Z/k)^2$, for $k$ as in Lemma \ref{lem:orth};
\item
if $k = 1$ then $d(M)$ (which divides $24$) is the greatest common
divisor of $\bar c_2(Z_+)A_+$ and $\bar c_2(Z_-)A_- \in \Z/24$, for
$\bar c_2(Z_\pm)$ as in \eqref{eq:c2blowup};
\begin{samepage}
\item
if $k = 1$ and $d(M)$ is divisible by $8$
(so that the Eells-Kuiper invariant $\EK(M)$ takes values in~$\Z/2$) then
\[ \EK(M) = \frac{\gd(\bar c_2(Z_+))\gd(\bar c_2(Z_-))}{4} \in \Z/2 , \]
where $\gd(\bar c_2(Z_\pm)) \in \{2,4,6,8,12,24\}$ is the greatest divisor of
$\bar c_2(Z_\pm)$.
\end{samepage}
\end{enumerate}
\end{lem}

\orthtable

\begin{proof}
(i)-(iii) are immediate consequences of Theorem \ref{thm:tcs_H}.

The image $N'_\mp$ of the product homomorphism
$\flat^{\pm} : N_\mp \to N_\pm^*$ from \eqref{eq:flat} has rank 1,
and Lemma \ref{lem:orth} implies that it has cotorsion $\Z/k$ in $N_\pm^*$.
(iv) now follows from \eqref{eq:h4tor}.

When $k = 1$, so that $N'_\mp \subset N_\pm^*$ is primitive, %
the isomorphism $N_\pm^*/N'_\mp \cong \Z$ is realised by evaluation on $A_\pm$.
Therefore (v) follows from Proposition \ref{prop:p1y} and
Lemma \ref{lem:imy}.

\enlargethispage{0.1\baselineskip}
If $x_\pm \in N_\pm$ such that $\flat^\pm(x_\pm) = \bar c_2(Z_\mp) \mod d(M)$,
then the image of $\frac{x_\pm}{2}$ in $N_\pm/\gen{A_\pm}$ has
the same parity as $\frac{\gd(\bar c_2(Z_\pm))}{2}$.
Since $k = 1$, we find
\[ \frac{x_+}{2} . \frac{x_-}{2} \; = \;
\frac{\gd(\bar c_2(Z_+))}{2} \, \frac{\gd(\bar c_2(Z_-))}{2} \mod 2 , \]
and (vi) follows from Theorem \ref{thm:ek_tcs}.
\end{proof}

Six examples of matchings of rank 2 Fanos with non-trivial intersection
$N_+ \cap N_-$ are listed in \cite[Example No 9]{g2m}. However, Lemma
\ref{lem:orth} and the data in Table \ref{table:orth} allow us to be more
decisive. 
 
\begin{thm}
\label{thm:orth}
There are precisely 19 pairs of rank 2 Fanos that can be matched to
give twisted connected sums with $H^2(M) \cong \Z$. In all cases $H^4(M)$ is
torsion-free and $\EK(M) = 0$. For each of the pairs there is at least one
matching such that $\pi_2(M) \cong \Z$.
\end{thm}

\begin{samepage}
\noindent
In Table \ref{table:orth} we list the following data about the 19 pairs:
\begin{itemize}
\item the numbers $\#_\pm$ of the entries in the Mori-Mukai list used,
\item the square of the generator $B$ of the intersection $N_+ \cap N_-$,
\item the squares of the ample classes $A_\pm \in N_\pm$,
\item the topological invariants $b_3(M)$ and $d(M)$ of the resulting
twisted connected sums.
\end{itemize}
\end{samepage}

\footnotetext{We thank Guio, Jockers, Klemm and Yeh \cite{GJKY18}
for bringing to our attention that this example was omitted in a previous
version of this paper.}

\begin{proof}[Proof of Theorem \ref{thm:orth}.]
In Table \ref{table:blocks} we have listed $h := \log_2 \frac{\Delta}{A^2}$ and
$B^2$ for all ample classes $A$ in the Picard lattices of rank 2 Fanos, such
that $h$ is not too small. Therefore it is easy to read off all cases where
the criterion of Lemma \ref{lem:orth} is satisfied, and they are the ones
listed in Table \ref{table:orth}.
Indeed, because $\Delta$ is never greater than $2A^2$ for any entry in the
table, matchings of rank 2 Fano 3-folds are only possible when
$\Delta_+\Delta_- = A_+^2A_-^2$, or equivalently $h_+ + h_- = 0$.

For each of the 19 pairs, we can apply Nikulin \cite[Theorem 1.12.4]{nikulin79}
to embed the rank 3 lattice $W$ from Lemma \ref{lem:orth} in $L$, thus defining
a configuration of primitive embeddings $N_+, N_- \subset L$. Because
$\Lambda_\pm = N_\pm$ we can apply 
Propositions  \ref{prop:match} and \ref{prop:generic} 
to find a compatible matching.

Since we always have $k = 1$, any twisted connected sum arising from matchings
with non-trivial $N_+ \cap N_-$ must have $\Tor \pi_2(M) \cong \Tor H^3(M) = 0$
by Lemma \ref{lem:orth_top}(iii). For the one pair (row 14 in
Table \ref{table:orth}) where $d(M)$ is divisible by 8, so that $\EK(M)$ could
possibly be non-zero, Lemma \ref{lem:orth_top}(v) shows that $\EK(M) = 0$ anyway.
\end{proof}

There are two reasons why Theorem \ref{thm:orth} does not claim that there
is a unique diffeomorphism type of twisted connected sum arising from each
of the 19 pairs. The first is that for some of the pairs we can also embed
$W \subset L$ non-primitively, giving a twisted connected sum $M$ where
$\Tor \pi_2(M) \cong \Tor H^3(M)$ is non-trivial; we do not study this
further at this point.

The second reason is that even if we consider only the
matchings with $\pi_2(M) \cong \Z$, we cannot automatically deduce that the
resulting diffeomorphism type is independent of the matching. This is because
we do not have a classification theorem for this type of manifold. 
We hope to return to the problem of
classifying such manifolds elsewhere. We expect that to determine the 
diffeomorphism type one must further compute the square of a generator of
$H^2(M)$, and also some generalisation of the invariants used by Kreck and
Stolz \cite{kreck91} for 7-manifolds with $\pi_2(M) \cong \Z$
and \emph{finite} $H^4(M)$.

Similarly, the computations presented here are not sufficient to decide 
whether the two entries in Table \ref{table:orth} that have the same values of
$b_3(M)$ and $d(M)$ (rows 2 and 8) must be diffeomorphic or even homeomorphic.
However, according to Wallis \cite[\S B.3]{wallis:thesis} these two manifolds are
distinguished by the square of the generator of $H^2(M)$.

\section{Skew matching}
\label{sec:handcraft}

Having dealt with matchings of rank 2 Fano 3-folds where the configuration
$N_+, N_- \subset L$ is orthogonal (whether $N_+ \cap N_-$ is trivial or not),
we now consider the skew case.
To find skew matchings we have to deal with lattice-arithmetical problems, and
also to gain some detailed understanding of the
deformation theory of the Fanos involved. Because of the case-by-case checking
required, we think of this task as `handcrafting'.

\subsection{Arithmetic conditions for skew matching}

For Fano 3-folds of Picard rank 2, at least the arithmetic part of the
problem of finding matchings with a skew configuration $N_+, N_- \subset L$
can be dealt with systematically.
As pointed out in Remark \ref{rmk:trichotomy}, $N_+ \cap N_-$ must be
trivial in this case, and the subgroup $W_\pm \subset N_\pm$ orthogonal to
$N_\mp$ has rank exactly 1, and is generated by some ample class $A_\pm$.
Then $W$ is isometric to $W_k$
for some integer $k$, where the quadratic form on
\[W_k := N_+ \oplus N_-\] is characterised as follows: $A_\pm \perp N_\mp$,
and $\bm_+.\bm_- = k$ for some $\bm_\pm \in N_\pm$ such that $\bm_\pm, A_\pm$
is a basis for $N_\pm$. (The choice of $H_\pm$ only affects the sign of $k$.)

\begin{lem}
\label{lem:handarith}
Let $-\Delta_\pm$ be the discriminant of $N_\pm$. Then
$W_k$ has signature $(2,2)$ if and only if
\begin{equation*}
k^2A_+^2A_-^2 < \Delta_+\Delta_-.
\end{equation*}
\end{lem}

\begin{proof}
Since $W_k$ contains the positive definite subspace $\gen{A_+, A_-}$ and
some negative elements, its signature is (2,2) if and only if its 
discriminant $D$ is positive.
Let $B_\pm \in N_\pm$ be a generator for the orthogonal complement of
$A_\pm$ in $N_\pm$, %
and let $n_\pm$ be the index of $\gen{A_\pm, B_\pm} \subseteq N_\pm$.
The index $n_+n_-$ sublattice of $W$ spanned by $A_+, B_+, A_-, B_-$ has
intersection form
\[ \begin{pmatrix}
A_+^2 & 0 & 0 & 0 \\
0 & B_+^2 & 0 & kn_+n_- \\
0 & 0 & A_-^2 & 0 \\
0 & kn_+n_- & 0 & B_-^2
\end{pmatrix}  \]
and discriminant $A_+^2 A_-^2 (B_+^2B_-^2 - k^2n_+^2n_-^2) = n_+^2n_-^2D$.
Since $A_\pm^2 B_\pm^2 = -\Delta_\pm n_\pm^2$, we find
\[ D = \Delta_+\Delta_- - k^2A_+^2A_-^2 . \qedhere \] 
\end{proof}

In particular, a necessary condition for finding a matching of rank 2 Fanos
$Y_+, Y_-$ with skew configuration $N_+, N_- \subset L$
is that there are ample classes $A_\pm \in N_\pm$
with
\begin{equation}
\label{eq:hsum}
h_+ + h_- > 0,
\end{equation}
for $h_\pm := \log_2 \frac{\Delta_\pm}{A_\pm^2}$.
We can readily identify pairs satisfying this necessary condition from
Table \ref{table:blocks}.

In terms of the notation from \S \ref{subsec:config}, $W_\pm$ is generated
by $A_\pm$, and \emph{if} $W_k \subset L$ is primitive then $\Lambda_\pm$ is
the orthogonal complement of $A_\pm$ in $W_k$.
Let us make an observation about the form on $\Lambda_\pm$ that will prove
useful.

\begin{lem}
\label{lem:lambda}
Let $N_\pm$ be a pair of rank 2 lattices of signature $(1,1)$, and discriminant
$-\Delta_\pm$. Given positive classes $A_\pm \in N_\pm$, define $W_k$ as
in Lemma \ref{lem:handarith}, with $k > 0$ such that
$\Delta_+\Delta_- > k^2 A_+^2 A_-^2$. Let $B_\pm$ be a generator for
the orthogonal complement of $A_\pm$ in $N_\pm$, and let $\Lambda_\pm$ be
the orthogonal complement of $A_\mp$ in $W_k$.

\pagebreak[2]

Suppose $\bm \in N_\pm$ has the property that
\begin{equation}
\label{eq:discr_rel}
- H^2 B_\mp^2 (\Delta_+\Delta_- - k^2 A_+^2 A_-^2)
\; \geq \; \Delta_\pm \Delta_+ \Delta_- .
\end{equation}
Then
\[ (v.\bm)^2 - v^2 \bm^2 \; \geq \; \Delta_\pm  \]
for any $v \in \Lambda_\pm$ linearly independent of $\bm$.
\end{lem}

\begin{proof}
The inequality certainly holds if $v \in N_\pm$. If $H$ is a multiple of
$A := A_\pm$, then the inequality follows easily for any $v \in \Lambda_\pm$.

$\Lambda_\pm$ is generated by $N_\pm$ together with $B := B_\mp$.
If $H$ is linearly independent of $A$, then the projection $B'$ of $B$ to the
orthogonal complement of $N_\pm$ in $\Lambda_\pm$ can be written as
\[ B' \; = \; B - \frac{B.H}{(A.H)^2 - A^2 H^2}\left((A.H)A + A^2 H\right), \]
whose square is
\[ (B')^2 \; = \; B^2 - \frac{(B.H)^2A^2}{(A.H)^2 - A^2H^2} . \]
It suffices to show that
\[ -(B')^2H^2 \geq \Delta_\pm . \]
If we let $m$ denote the index of $\gen{A,H}$ in $N_\pm$, and let
$n$ be the index of $\gen{A_\mp, B}$ in $N_\mp$,
then $(A.H)^2 - A^2 H^2 = m^2\Delta_\pm$, while $B.H = \pm k mn$.
Therefore $(B.H)^2\Delta_\mp = -k^2 m^2 A_\mp^2 B^2$, and
\[ -\Delta_+\Delta_- H^2(B')^2 
\; = \; -H^2 B^2 (\Delta_+ \Delta_- - k^2 A_+^2 A_-^2) . \qedhere
\]
\end{proof}

\subsection{Topology of skew matchings}

We now explain how to determine the topology of twisted connected
sums obtained from a skew matching of rank 2 blocks.
In particular, we identify all pairs of deformation types of rank 2 Fano
3-folds that could possibly be matched to produce twisted connected sums with
non-zero generalised Eells-Kuiper invariant.

\begin{prop}
\label{prop:skew_top}
Let $(Z_\pm, \kd_\pm)$ be a pair of building blocks whose polarising lattices
$N_\pm$ have rank~2, such that the kernel of $H^2(Z_\pm) \to H^2(\kd_\pm)$
is generated by $[\kd_\pm]$.
Let $\hkr : \kd_+ \to \kd_-$ be a matching whose configuration
$N_+, N_- \subset L$ has $N_+ \oplus N_-$ isometric to $W_k$
from Lemma \ref{lem:handarith}, and let
$M$ be the resulting twisted connected sum.

Let $A_\pm \in N_\pm$ be a generator for the orthogonal
complement of $N_\mp$ in $N_\pm$. Then
\begin{enumerate}
\item $H^2(M) = 0$;
\item $b_3(M) = b_3(Z_+) + b_3(Z_-) + 21$;
\item $\Tor H^3(M) \cong L/W_k$;
\item \label{it:skew_tor}
$\Tor H^4(M) \cong (\Z/k)^2$;
\item
if $k = 1$ then $d(M)$ (which divides $24$) is the greatest common
divisor of $\bar c_2(Z_+)A_\pm$ and
$\bar c_2(Z_-)A_\pm \in \Z/24$, for
$\bar c_2(Z_\pm)$ as in \eqref{eq:c2blowup};
\item
if $k = 1$ and $d(M)$ is divisible by $8$
(so that the Eells-Kuiper invariant $\EK(M)$ takes values in~$\Z/2$) then
\[ \EK(M) = \frac{\gd(\bar c_2(Z_+))\gd(\bar c_2(Z_-))}{4} \in \Z/2 , \]
where $\gd(\bar c_2(Z_\pm)) \in \{2,4,6,8,12,24\}$ is the greatest divisor of
$\bar c_2(Z_\pm)$.
\end{enumerate}
\end{prop}

\begin{proof}
(i)-(iii) are immediate from Theorem \ref{thm:tcs_H}.
(iv)-(vi) are entirely analogous to the proof of Lemma \ref{lem:orth_top}.
\end{proof}

\begin{rmk}
We see in Table \ref{table:blocks} that $h < 1$ (\ie $\Delta <2A^2$) for
any ample class $A$ on any rank~2 Fano~3-fold.
Therefore Lemma \ref{lem:handarith} implies
that $W_k$ can only have signature $(2,2)$ for $k = 1$. 
Therefore, whenever we do find a matching of rank 2 Fanos with a skew
configuration, the resulting twisted connected sum $M$ always has $H^4(M)$
torsion-free by Proposition \ref{prop:skew_top}\ref{it:skew_tor}.
\end{rmk}

\begin{samepage}
\begin{rmk}
\label{rmk:candidates}
$\EK(M)$ is non-zero if and only if both $\bar c_2(Z_\pm)A_\pm$
are divisible by 8 while neither $\bar c_2(Z_\pm)$ is divisible by 4. 
Consulting Table \ref{table:blocks}, we find that the only rank 2 Fano type
blocks with such ample classes $A$ are
\begin{enumerate}[widest=\#27]
\item[\#9] with $A := \bl + \bm$ $\phantom{2} (h = 0.09)$, %
\item[\#17] with $A := \bl + \bm$ $\phantom{2} (h = 0.06)$,
\item[\#18] with $A := \bl + 2\bm$ $(h = -0.58)$,
\item[\#27] with $A := \bl + \bm$ $\phantom{2} (h = 0.09)$.
\end{enumerate}
Since \#18 has $h$ quite negative, the condition \eqref{eq:hsum} implies
that the only ways to match rank 2 Fanos
with a skew configuration to construct a twisted connected sum $M$ with
non-trivial $\EK(M)$ is to use a pair among \#9, \#17 and~\#27.
\end{rmk}
\end{samepage}

\subsection{Handcrafting examples with non-zero generalised Eells-Kuiper invariant}

Let $\sff_+$, $\sff_-$ be a pair of deformation types of rank 2 Fanos.
For a skew configuration $N_+, N_- \subset L$ of their Picard lattices,
$\Lambda_\pm$ has rank 3. Because $\Lambda_\pm$ is strictly bigger than
$N_\pm$, the genericity result Proposition \ref{prop:generic} does not suffice
for applying Proposition \ref{prop:match} to find a matching with the
prescribed configuration.
The laborious part of finding skew matchings is to
prove that generic $\Lambda_\pm$-polarised K3 surfaces still appear as
anticanonical K3 divisors in some members of $\sff_\pm$, despite being more
special than the generic, $N_\pm$-polarised K3 divisors in the family.

In the present paper we go to that effort only in cases that lead to
twisted connected sums with $\EK(M) \not= 0$. 
In Theorem \ref{thm:nonzero} we show that all skew configurations of the Picard
lattices of rank 2 Fanos identified in
Remark \ref{rmk:candidates} are in fact realised by a
matching.
The genericity results needed to find those matchings are provided by
Lemma \ref{lem:generic}. To prove that we use

\begin{lem}
\label{lem:embed}
Let $\kd$ be a K3 surface, and $\bm \in \Pic \kd$ a primitive nef class with
$\bm^2 \geq 4$.
Then $\bm$ is a very ample class (\ie the linear system $|\bm|$ defines an
embedding $\kd \into \PP^{\frac{\bm^2}{2} + 1}$) unless there is
some $v \in \Pic \kd$ such that
\begin{samepage}
\begin{enumerate}
\item $d = 2$ and $v^2 = 0$, or
\item $d = 0$ and $v^2 = -2$,
\end{enumerate}
for $d := v.\bm$.
\end{samepage}
\end{lem}

\begin{proof}
According to \cite[Lemma 7.15]{g2m} (based on \cite[Chapter 3]{reid97}), the
arithmetic conditions rule out all possible ways that $\bm$ can fail to
be very ample. (i) rules out the existence of any classes with $d = 1$ and $v^2 = 0$, and hence $|\bm|$ being monogonal. Therefore $|\bm|$ has no fixed
part, and defines a morphism. $\bm^2 \not= 2$, $\bm$ primitive and (i) rule out
the three different ways that $|\bm|$ could be hyperelliptic, so it defines a
birational morphism onto its image. (ii) rules out any $(-2)$-curves being
contracted, so $|\bm|$ defines an isomorphism.
\end{proof}

Using Lemma \ref{lem:embed} and various results about curves on K3 surfaces,
it is for many families $\sff$ of Fano 3-folds possible to obtain conditions
on a lattice $\Lambda$ (containing the Picard lattice $N$ of $\sff$) that
ensure that any K3 with Picard lattice isometric to $\Lambda$ can be embedded
as an anticanonical divisor in some element of $\sff$. The conditions are to
exclude the existence in $\Lambda$ of elements $v$ with certain combinations of
$v^2$ and inner products of $v$ with elements of $N$. Once these `handcrafting'
conditions have been proved for some collection of blocks, then it would be
feasible to get a computer to generate candidate configurations involving those
blocks, and to verify whether the handcrafting conditions hold.

However, for the purposes of checking a few examples by hand, it is expedient
to instead organise the argument by checking that in the examples we are
concerned with, the sort of inequality produced by Lemma \ref{lem:lambda} rules
out the presence in $\Lambda$ of vectors $v$ with the relevant properties.

\enlargethispage{0.2\baselineskip}
\begin{lem}
\label{lem:generic}
Let $\sff$ be the deformation type of \#27, \#9 or \#17
in the Mori-Mukai list of rank~2 Fanos. Let $N \subset L$ be a primitive
embedding of its Picard lattice, and let $\Lambda \subset L$ be a primitive
lattice containing $N$. Let $\bl, \bm$ be the basis of $N$ described in
\S \ref{subsec:MM}, and $\Amp_\sff \subset N(\bbr)$ the interior of
the cone spanned by $\bl$ and $\bm$. Suppose that
\begin{equation}
\label{eq:hyp}
(v.H)^2 - H^2v^2 \geq \Delta
\end{equation}
for all $v \in \Lambda$ that are linearly independent of $H$, where
$-\Delta$ is the discriminant of $N$.
Then $\sff$ is $(\Lambda, \Amp_\sff)$-generic.
\end{lem}

\begin{proof}
It is enough to prove that any K3 surface $\kd$ with Picard lattice
isometric to \emph{precisely}~$\Lambda$
has a marking $h : H^2(\kd) \to L$ mapping $\Pic(\kd) \to \Lambda$
and an embedding $\kd \into Y$ of $\kd$ as an anticanonical divisor
in some $Y \in \sff$, such that the image of the ample cone of $Y$ in
$H^2(\kd)$ is $h^{-1}(\Amp_\sff)$.
\begin{enumerate}[itemindent=11mm,labelsep=3mm,leftmargin=0mm,itemsep=2mm]
\item[\#27]
This is the blow-up of $\PP^3$ in a twisted cubic.
In this case the Picard lattice is $\sm{2 & 5 \\ 5 & 4}$, so $H$
is a class of degree 4. \eqref{eq:hyp} means that if $v \in \Lambda$ %
is linearly independent of $H$ and $d := v.H$ then
\begin{equation}
\label{eq:deg_bound}
d^2 - 4v^2 \geq 17,
\end{equation}
so neither of the cases in Lemma \ref{lem:embed} can occur.
Since $H$ is not orthogonal to any $(-2)$-class in~$\Lambda$, we can choose a
marking of $\kd$ that maps $\Pic(\kd)$ to $\Lambda$, such that $\bm$ is the
image of a nef class. Lemma \ref{lem:embed} then implies that $\bm$ corresponds
to a very ample class, embedding $\kd$ as a quartic in~$\PP^3$.

The class $2\bm - \bl$ has square $-2$, so by a standard application of the
Riemann-Roch theorem for surfaces either $2\bm-\bl$ or $-(2\bm-\bl)$
is effective \cite[Corollary 3.7(i)]{reid97}. Since $\bm.(2\bm-\bl) = 3$ is
positive, it must be $2\bm-\bl$ that is effective.
Since any irreducible effective class $v$ has $v^2 \geq -2$,
\eqref{eq:deg_bound} implies that in fact there are no such classes of degree
$d \leq 2$.
Therefore $2\bm-\bl$ is irreducible.
A further well-known application of Riemann-Roch
(see discussion after \mbox{\cite[Corollary 3.7]{reid97}})
implies that $2\bm-\bl$ is represented by a smooth rational curve $\Gamma$.
Its image in $\PP^3$ has degree~3, so must be a twisted cubic.
By blowing up $\Gamma$ we obtain a Fano $Y \in \sff$.

\item[\#9]
This is the family of blow-ups of $\PP^3$ in smooth curves of
degree 7 and genus 5.
The Picard lattice is $\sm{2 & 5 \\ 5 & 4}$ in this case too,
so we have already proved that $\bm$ embeds $\kd$ as a quartic K3 in $\PP^3$,
and that $2\bm-\bl$ is irreducible and represented by a twisted cubic $\Gamma$.

We now want to prove that $3\bm-\bl = \bm + [\Gamma]$ can be represented by a
smooth curve $C$. %
Because $3\bm -\bl$ is big and nef, by \cite[Theorem 3.8(d) \& 3.15]{reid97}
the only way it can fail to be basepoint-free is if it is monogonal,
\ie if $3\bm-\bl = aE + F$ where $E^2 = 0$ and $F$ is the fixed part.
Since $\bm$ is very ample certainly $\bm + [\Gamma]$ cannot have any
non-trivial fixed part other than $\Gamma$, and since $\bm^2 \not= 0$ we cannot
have $F = \Gamma$ either.

Hence $|3\bm-\bl|$ is basepoint-free, and a general member $C$ is
smooth by Bertini's theorem.
Now $C$ has degree 7 and genus 5, so by Blanc-Lamy \cite[Theorem 1.1]{blanc12}
blowing up $C$ defines a Fano in the family~$\sff$ unless there is a
4-secant line $\ell$ to $C$.

It thus remains only to show that no such $\ell$ exists.
If it did, then restricting the pencil of hyperplanes containing $\ell$
to $C$ defines a pencil of degree 3, a ``complete $g^1_3$'' in the terminology
of Brill-Noether theory.
The curve $C$ therefore has Clifford index 1. By the main theorem of
Green-Lazarsfeld \cite{green87} this implies
the existence of a line bundle $L \in \Lambda$ with $L^2 = 0$ and $L.C = 3$.
Now \eqref{eq:deg_bound} implies $H.L \geq 5$.
But then the span of $H+L$ and $3H-G$ is positive-definite, which is
impossible.

\item[\#17]
This is the family of blow-ups of smooth quadrics $Q \subset \PP^4$ in
smooth elliptic curves of degree~5. 
In this case the Picard lattice is $\sm{4 & 7 \\ 7 & 6}$, so $H$ is
a degree 6 class. \eqref{eq:hyp} means that
\begin{equation}
\label{eq:deg_bound2}
d^2 - 6v^2 \geq 25
\end{equation}
for any $v \in \Lambda$ not a multiple of $H$. In particular
Lemma \ref{lem:embed} implies that $\bm$ is very ample,
and by Riemann-Roch the embedded image of $\kd$ in $\PP^4$ is the intersection
of a quadric $Q$ and a cubic.
The quadric must be smooth, for if $Q$ is singular then it contains planes.
The section of $\kd$ by such a plane would be a plane cubic, defining a class
with $v^2 = 0$ and $d = 3$, which is ruled out by \eqref{eq:deg_bound2}
(see Fukuoka \cite[Lemma 2.4]{fukuoka17}).

Now consider the class $E := 2\bm- \bl$, which has degree 5 and $E^2 = 0$.
\eqref{eq:deg_bound2} rules out the existence of irreducible classes
in $\Pic \kd$ with $d \leq 3$, so $E$ is irreducible. In particular $E$ does
not have any $(-2)$-curve components, so $E$ is nef.
Therefore \cite[Theorem 3.8(b)]{reid97} implies that $|E|$ is basepoint-free.
A generic $C \in |E|$ is therefore a smooth elliptic curve of degree 5.

If $C$ had a 4-secant line $\ell$, then the pencil of planes through $\ell$
would define a morphism from $C$ to $\PP^1$, which is impossible.
Therefore by Blanc-Lamy \cite[Theorem 1.1]{blanc12},
blowing up $Q$ in $C$ yields a member $Y \in \sff$.
\qedhere
\end{enumerate}
\end{proof}

\noindent
For the three families of blocks under consideration, the sufficient conditions
provided by Lemmas \ref{lem:lambda} and \ref{lem:generic} turn out to be good
enough to prove the existence of skew matching between any pair of families.
Moreover, an ad hoc argument rules out any configurations beyond the ones
realised.

\begin{thm}
\label{thm:nonzero}
Let $\sff_+$ and $\sff_-$ be a pair of deformation types of rank 2 Fanos
among \#9, \#17 and \#27 on the Mori-Mukai list.
Let $N_\pm$ be their Picard lattices, and let $A_\pm \in N_\pm$ be the ample
class listed in Table \ref{table:blocks}.
Define $W_1$ as in Lemma \ref{lem:handarith}, embed $W_1 \subset L$
primitively, and consider the resulting configuration $N_+, N_- \subset L$. 
\begin{enumerate}
\item There is a matching of some elements of $\sff_\pm$ with that
configuration.
\item This is the only non-perpendicular configuration of $N_+$ and $N_-$
for which a matching exists.
\end{enumerate}
\end{thm}

\pagebreak[3]

\begin{proof}%
\begin{enumerate}[itemindent=8mm,labelsep=3mm,leftmargin=0mm,itemsep=2mm]
\item 
For this configuration, the lattice $\Lambda_\pm \subset L$ defined in
\S \ref{subsec:config} is the orthogonal complement of $A_\mp$ in $W_1$.
Let $G_\pm, H_\pm \in N_\pm$ be the basis vectors
described in \S \ref{subsec:MM}. Looking up the values of $H_\pm^2$, $A_\pm^2$,
$B_\pm^2$ and $\Delta_\pm$ in Table \ref{table:blocks}, Lemma \ref{lem:lambda}
implies that
\begin{equation}
\label{eq:discr_pm}
(v.H_\pm)^2 - H_\pm^2v^2 \geq \Delta_\pm
\end{equation}
for any $v \in \Lambda_\pm$ linearly independent of $H_\pm$.
Therefore Lemma \ref{lem:generic} implies that $\sff_\pm$ is
$(\Lambda_\pm, \Amp_{\sff_\pm})$-generic, and the desired matching exists by
Proposition \ref{prop:match}.

\item
By Lemma \ref{lem:orth} there can be no matchings of these types with
a configuration such that $N_+ \cap N_-$ is non-trivial, so a 
non-perpendicular configuration must be skew. We have explained that for a skew
configuration to satisfy the conditions of Remark \ref{rmk:arith}, $N_+ \oplus
N_-$ must be isometric to~$W_1$. Thus it only remains to rule out %
configurations where $W_1$ is embedded non-primitively in~$L$.

In Lemma \ref{lem:handarith} we computed that the discriminant of $W_1$ is
\[ D = \Delta_+\Delta_- - A_+^2 A_-^2 . \]
When $\sff_\pm$ are both among \#9 and \#27, $D = 33$ is square-free, so
$W_1$ does not have any integral overlattice.
When $\sff_+$ is one of \#9 and
\#27 while $\sff_-$ is \#17 we get $D = 41$, which is also square-free.
Hence there are no non-primitive embeddings $W_1 \subset L$ in these cases.

However, when $\sff_\pm$ are both \#17 we get $D = 49$.
In the basis $G_+, A_+, G_-, A_-$,
the quadratic form on $W_1$ can be written as
\vspace{-0.5\baselineskip}
\[ \begin{pmatrix}
4 & 11 & 1 & 0 \\ 11 & 24 & 0 & 0 \\ 1 & 0 & 4 & 11 \\ 0 & 0 & 11 & 24
\end{pmatrix} . \]
This matrix has rank 3 over $\Z/7$, so
the discriminant group must be $\Z/49$ rather than $(\Z/7)^2$,
and $W_1$ has a index 7 overlattice $\wt W$ (which is in fact unimodular).
Indeed, $K := G_+ + A_+ - G_- - A_-$ has $K^2 = 98$ and its product with
any element of $W$ is divisible by 7. Therefore we can define $\wt W$ by
adjoining $\frac{1}{7}K$ to $W$.

The only possible
way to embed $W_1 \subset L$ non-primitively is via
a primitive embedding $\wt W \subset L$. We now check that there are
no matchings with this configuration.
Note that $\Lambda_\pm$ is spanned by
$G_\pm$, $H_\pm$ and $\wt B_\pm := \pm \frac{24}{7} K + 5A_\mp$.
In that basis, the quadratic form on $\Lambda_\pm$ is represented by
\vspace{-0.5\baselineskip}
\[ \begin{pmatrix}
4 & 7 & 48 \\ 7 & 6 & 72 \\ 48 & 72 & 552
\end{pmatrix} . \]
We find that \eqref{eq:deg_bound2} fails for some $v \in \Lambda_+$ such that
$v^2 = 0$ and $d = 3$, \eg $v = E' := - 9G_+ - H_+ + \wt B_+$.

Now suppose that $\kd$ is an anticanonical K3 divisor in some rank 2 Fano of
type \#27, with $\Pic \kd$ isometric to this $\Lambda_+$.
Then $H_+$ embeds $\kd$ in a smooth quadric $Q \subset \PP^4$.
The argument from the proof of Lemma \ref{lem:generic} for the case \#17
shows that because $(E')^2 = 0$ and $E'.H_+ = 3$, the linear system $|E'|$
on $\kd$ is basepoint-free, and represented by a smooth elliptic curve $C$. By
Riemann-Roch, $C$ is a plane cubic. Because $Q$ contains $C$ it must also
contain the plane of $C$, contradicting that $Q$ is non-singular.
\qedhere
\end{enumerate}
\end{proof}

\subsection{Proof of main theorem}

To prove Theorem \ref{thm:main} it now remains only to put together the pieces
provided above.

Theorem \ref{thm:nonzero} provides exactly one configuration with a matching
for each pair of rank 2 Fano types among \#9, \#17 and \#27 (referred to
as \ref{it:mm9}, \ref{it:mm17} and \ref{it:mm27} in Table \ref{table:mainex} in
the introduction).
Each of those six pairs produces a closed 7-manifold $M$ with holonomy $G_2$ by
Construction \ref{constr:tcs}, whose topology can be computed
from Proposition \ref{prop:skew_top} and the data in Table \ref{table:blocks}.
These $M$ are 2-connected with $H^4(M)$ torsion-free, and $b_3(M)$ and $d(M)$
as listed in Table \ref{table:mainex}. By design $\EK(M) = 1$, while
all other matchings of rank 2 Fanos give $\EK = 0$ by Corollary \ref{cor:perp0},
Theorem \ref{thm:orth} and Remark~\ref{rmk:candidates}.

The six matchings realise four different pairs $(b_3,d)$,
and hence four different diffeomorphism types by Theorem \ref{thm:2c7m}(ii).
We can then consult Table \ref{table:perp} to see that for two of
these four smooth manifolds, there exist perpendicular twisted connected
sums $M'$ with the same $(b_3, d)$.
Then $M$ and $M'$ are homeo\-morphic by Theorem \ref{thm:2c7m}(i).
However, $\EK(M') = 0$ by Corollary \ref{cor:perp0}, so $M$ and $M'$ are not
diffeomorphic.

\bibliographystyle{amsinitial}
\bibliography{g2geom}
\enlargethispage{0.5\baselineskip}

\end{document}